\documentclass[12pt,reqno]{amsart}
\usepackage{}
\usepackage{amssymb}
\usepackage{mathrsfs}
\usepackage{palatino}
\usepackage{bbm}

\newcommand{\FF}{\mathbb{F}}

\newcommand{\NN}{\mathbb{N}}

\newcommand{\ZZ}{\mathbb{Z}}

\newcommand{\yp}{Y^{[p]}_2}

\newcommand{\fa}{\mathfrak{a}}

\newcommand{\fb}{\mathfrak{b}}

\newcommand{\fg}{\mathfrak{g}}

\newcommand{\fh}{\mathfrak{h}}

\newcommand{\fl}{\mathfrak{l}}
\newcommand{\fm}{\mathfrak{m}}

\newcommand{\fn}{\mathfrak{n}}
\newcommand{\fp}{\mathfrak{p}}

\newcommand{\ft}{\mathfrak{t}}

\newcommand{\kk}{\mathbbm{k}}
\newcommand{\gln}{\mathfrak{gl}_N}
\newcommand{\gl}{\mathfrak{gl}}

\newcommand{\ip}{\boldsymbol{I}_{\boldsymbol{p}}}
\newcommand{\inn}{\boldsymbol{I}_{\boldsymbol{\mathbb{N}}}}
\newcommand{\zpl}{Z^{{[p]}}(\lambda(u))}

\DeclareMathOperator{\Mod}{mod}
\DeclareMathOperator{\ev}{ev}
\DeclareMathOperator{\GL}{GL}

\DeclareMathOperator{\diag}{diag}

\DeclareMathOperator{\Hom}{Hom}

\DeclareMathOperator{\row}{row}
\DeclareMathOperator{\col}{col}
\DeclareMathOperator{\Lie}{Lie}

\usepackage{pictex}
\usepackage{amscd}
\usepackage{latexsym}
\usepackage{amssymb}
\usepackage{eucal}
\usepackage{eufrak}
\usepackage{verbatim}
\usepackage{bbm}
\usepackage{hyperref}
\hypersetup{colorlinks,linkcolor=blue,urlcolor=cyan,citecolor=red}

\numberwithin{equation}{section}

\newtheorem{Theorem}{Theorem}[section]
\newtheorem{Lemma}[Theorem]{Lemma}
\newtheorem{Corollary}[Theorem]{Corollary}
\newtheorem{Proposition}[Theorem]{Proposition}

\theoremstyle{Theorem}

\newtheorem*{thm*}{Theorem A}
\newtheorem*{thm**}{Theorem B}
\theoremstyle{remark}
\newtheorem{Remark}{Remark}

\newtheorem*{Definition}{Definition}
\newtheorem*{Example}{Example}

\numberwithin{equation}{section}

\begin{document}

\title{Modular representations of the Yangian $Y_2$}

\author{Hao Chang, Jinxin Hu \lowercase{and} Lewis Topley}
\address[H. Chang]{School of Mathematics and Statistics, and Key Laboratory of Nonlinear Analysis \& Applications (Ministry of Education), Central China Normal University, Wuhan 430079, China}
\email{chang@ccnu.edu.cn}
\address[J. Hu]{School of Mathematics and Statistics, and Hubei Key Lab--Math. Sci., Central China Normal University, Wuhan 430079, China}
\email{hjx2021112103@mails.ccnu.edu.cn}
\address[L. Topley]{Department of Mathematical Sciences, University of Bath, Bath, BA2 7AY, UK}
\email{lt803@bath.ac.uk}

\date{\today}

\makeatletter
\makeatother


\begin{abstract}
Let $Y_2$ be the Yangian  associated to the general linear Lie algebra $\mathfrak{gl}_2$,
defined over an algebraically
closed field $\kk$ of characteristic $p>0$.
In this paper, we study the representation theory of the restricted Yangian $\yp$.
This leads to a description of the representations of $\mathfrak{gl}_{2n}$,
 whose $p$-character is nilpotent with Jordan type given by a two-row partition $(n, n)$.

\end{abstract}

\maketitle

\section{Introduction} \label{S:Intro}
For each simple finite-dimensional Lie algebra $\fa$ over the field of complex numbers, 
the corresponding Yangian $Y(\fa)$ was defined by Drinfeld in \cite{D85} as a canonical deformation of the universal
enveloping algebra $U(\fa[x])$ for the current Lie algebra $\fa[x]$. 
The Yangian $Y_N:=Y(\gl_N)$ associated to the Lie algebra $\mathfrak{gl}_N$ was earlier considered
in the work of the  St.-Petersburg school of mathematical physicists around 1980; see for instance \cite{TF79}.
Finite dimensional irreducible representations of the Yangians were classified by Drinfeld \cite{D88}.
His classification was preceded by the work of Tarasov \cite{Ta84,Ta85} which dealt with the case $N = 2$, 
see also \cite{Mol98}.
In the current article, we initiate a study of representations of the Yangian $Y_2$,
over an arbitrary field $\kk$ of positive characteristic $p>0$.

Aside from the intrinsic value of making headway in this new research topic, we are motivated by applications to the classical problem of describing representations of the general linear algebra $\gl_{2n}$.
Over the field of complex numbers,
Brundan–Kleshchev \cite {BK06} introduced 
shifted Yangians and their truncated analogues,
which are isomorphic to the {\it finite $W$-algebras} associated to nilpotent orbits in $\gl_N$,
as defined by Premet \cite{Pre02}.
Premet's original motivation for introducing $W$-algebras came from the representation theory of Lie algebras in positive characteristic.
He discovered some remarkable finite dimensional {\it restricted finite $W$-algebras},
which are Morita equivalent to the reduced enveloping algebras of modular reductive Lie algebras.
In \cite{BT18}, Brundan and the third author developed the theory of the shifted Yangian $Y_N(\sigma)$ over $\kk$. 
In particular, they gave a description of the centre $Z(Y_N(\sigma))$ of $Y_N(\sigma)$. One of the key
features which differs from characteristic zero is the existence of a large central subalgebra $Z_p(Y_N(\sigma))$, called the {\it $p$-center}.
Moreover, in their paper \cite{GT19},
the authors observed that Brundan–Kleshchev's isomorphism can be reduced into to positive characteristic,
i.e., the modular finite $W$-algebra is a truncation of the modular shifted Yangian.
Both algebras admit a $p$-centre, 
and the restricted version of Brundan–Kleshchev's isomorphism was established by Goodwin and the third author (see \cite[Theorem 1.1]{GT21}).
Thus Premet's restricted finite $W$-algebra in type {A} can be constructed as a restricted modular shifted Yangian.
Assembling these ingredients,
we may understand the representations of certain reduced enveloping algebras from the representation theory of Yangians.

As a first step towards developing the representation theory of modular Yangian,
we investigate the finite dimensional irreducible representations of the restricted Yangian $\yp$ in detail.
By definition, the {\it restricted Yangian} $\yp:=Y_2/Y_2Z_p(Y_2)_+$ is the quotient of $Y_2$ by the ideal generated by the generators of the $p$-center $Z_p(Y_2)$ (see Subsection \ref{S:Restricted Yangian}).
Fix $n\in\ZZ_{\geq 1}$. Let $\gl_{2n}$ be the general linear Lie algebra consisting of $2n\times 2n$-matrices and $e\in \gl_{2n}$the  $2\times n$-rectangular nilpotent element (see Subsection \ref{Restricted finite W-algebras}).
Our applications will rely on the following isomorphism (\cite[Theorem 1.1]{GT21}):
\begin{align}\label{gtiso}
Y_{2,n}^{[p]}\overset{\sim}{\longrightarrow} U^{[p]}(\gl_{2n},e),
\end{align}
where $Y_{2,n}^{[p]}$ is the restricted truncated Yangian of level $n$ and $U^{[p]}(\gl_{2n},e)$ is the restricted finite $W$-algebra associated to $e$ (see \cite[\S 4]{GT21}).
Recently,  the modular representations of Lie algebras of type $A$ with a two-row nilpotent $p$-central character have been studied under the assumption that the $p \gg 0$ (see \cite{DNY24}), and it would be interesting to compare our work with theirs.

The main results of this article are as follows.
\begin{enumerate}
\setlength{\itemsep}{6pt}
\item Theorem~\ref{Theorem:yp is hopf} shows that the Hopf structure on $Y_2$ descends to $Y_2^{[p]}$. In particular, $\yp$ should be viewed as the first Frobenius kernel of the quantum group $Y_2$.

\item Theorem~\ref{fd irreducible rep} classifies the finite dimensional simple $Y_2^{[p]}$-modules in terms of restricted highest weights;
\item Theorem~\ref{two poly are tensor} reinterprets these modules as certain tensor products of evaluation modules;
\item Theorem~\ref{Drinfeld poly} gives a necessary and sufficient condition for a simple highest weight module to be finite dimensional, in terms of Drinfeld polynomials;
\item Theorem~\ref{T:modgllofa} gives a combinatorial classification of the simple $U_\chi(\gl_{2n})$-modules with $p$-character $\chi$ of Jordan type $(n,n)$, as well as closed formulas for the dimensions of these simple modules.
\end{enumerate}

The article is organised as follows. In Section \ref{S:modular Yangian},
we recall some preliminaries about the modular Yangian $Y_2$ and its $p$-center.
In Section \ref{S:rep of Y2}, we define the baby Verma modules for $\yp$ and prove that 
every finite dimensional irreducible representation is isomorphic to the simple head of some baby Verma module.
We also give the necessary and sufficient condition for an irreducible representation to be finite dimensional. 
In Section \ref{Section:W-algebras}, we recap the theory of finite $W$-algebras, with a view to applications.
Using \eqref{gtiso}, we provide a new description of irreducible $Y_{2,n}^{[p]}$-modules from Levi subalgebras of $\gl_{2n}$.
In Section~\ref{Section:modforREA} go on to determine the irreducible modules for reduced enveloping algebras $U_\chi(\gl_{2n})$
and their dimensions, where $\chi$ is the $p$-character corresponding to $e$.\vspace{4pt}

\emph{Throughout this paper, $\kk$ denotes an algebraically closed field of characteristic $p>0$}.

\bigskip
\noindent
\textbf{Acknowledgment.} The authors would like to thanks Simon Goodwin and Adam Thomas for useful conversations and correspondence,
and for sharing details of their forthcoming work on this topic.
The first author is supported partially by NSFC (Grant Nos. 12071168, 12461005),
the Fundamental Research Funds for the Central Universities (Nos. CCNU22QN002, CCNU24JC001),
China Scholarship Council (No. 202306770031) and Key Laboratory of MEA, Ministry of Education.
He would like to thank Department of Mathematical Sciences at University of Bath
for their hospitality and support.
The third author is supported by a UKRI Future Leaders Fellowship, grant numbers MR/S032657/1, MR/S032657/2, MR/S032657/3.

\section{Restricted Yangian} \label{S:modular Yangian}
In \cite{BT18}, Brundan and the third author developed the theory of the Yangian $Y_N$ over a field of positive
characteristic. We only need here the special case of  $N=2$.

\subsection{Modular Yangian $Y_2$}
The Yangian associated to the general linear Lie algebra $\mathfrak{gl}_2$, denoted by $Y_2$,
is  the associated algebra over $\kk$ with the {\it RTT generators} $\{t_{i,j}^{(r)};~1\leq i,j\leq 2, r>0\}$ subject the following relations:
\begin{align}\label{RTT relations}
\left[t_{i,j}^{(r)}, t_{k,l}^{(s)}\right] =\sum_{t=0}^{\min(r,s)-1}
\left(t_{k, j}^{(t)} t_{i,l}^{(r+s-1-t)}-
t_{k,j}^{(r+s-1-t)}t_{i,l}^{(t)}\right)
\end{align}
for $1\leq i,j,k,l\leq 2$ and $r,s>0$.
By convention, we set $t_{i,j}^{(0)}:=\delta_{i,j}$.
We often put the generators $t_{i,j}^{(r)}$ for all $r\geq 0$ together to form the generating function
\[
t_{i,j}(u):= \sum_{r \geq 0}t_{i,j}^{(r)}u^{-r} \in Y_2[[u^{-1}]].
\]
It is easily seen that, in terms of the generating series,
the initial defining relation (\ref{RTT relations}) may be rewritten as follows:
\begin{align}\label{tiju tklv relation}
(u-v)[t_{i,j}(u),t_{k,l}(v)]=t_{k,j}(u)t_{i,l}(v)-t_{k,j}(v)t_{i,l}(v).
\end{align}

We need another set of generators for  $Y_2$ called {\it Drinfeld generators}. 
To define these, we consider the Gauss factorization $T(u)=F(u)D(u)E(u)$ of the matrix
$$
T(u) := \left(
\begin{array}{cc}
t_{1,1}(u)& t_{1,2}(u) \\
t_{2,1}(u)& t_{2,2}(u)
\end{array}
\right).
$$
This defines power series $d_i(u), e(u), f(u)\in Y_2[[u^{-1}]]$ such that 
$$
D(u)= \left(
\begin{array}{cc}
d_1(u)& 0 \\
0 & d_2(u)
\end{array}
\right),
E(u)= \left(
\begin{array}{cc}
1& e(u) \\
0 & 1
\end{array}
\right),
F(u)= \left(
\begin{array}{cc}
1& 0 \\
f(u) & 1
\end{array}
\right),
$$
Then we have that 
\begin{align}\label{t1122 by d}
t_{1,1}(u)=d_1(u), \, \,  \, \,  t_{2,2}(u)=f(u)d_1(u)e(u)+d_2(u),
\end{align}
\begin{align}\label{t1221 by d}
t_{1,2}(u)=d_1(u)e(u), \, \, \, \,  t_{2,1}(u)=f(u)d_1(u).
\end{align}
The Drinfeld generators are the elements $d_i^{(r)}, e^{(r)}$ and $f^{(r)}$ of $Y_2$ defined from the expansions
$d_i(u)=\sum_{r\geq 0}d_i^{(r)}u^{-r}, e(u)=\sum_{r>0}e^{(r)}u^{-r}$ and $f(u)=\sum_{r>0}f^{(r)}u^{-r}$.
Also define $d_i'^{(r)}$ from the identity $d_i'(u)=\sum_{r\geq 0}d_i'^{(u)}u^{-r}=:d_i(u)^{-1}$.

\begin{Theorem}\cite[Theorem 4.3]{BT18}
The algebra $Y_2$ is generated by the elements $\{d_i^{(r)}, d_i'^{(r)};~1\leq i\leq 2, r>0\}$ and $\{e^{(r)}, f^{(r)};~r>0\}$
subject to the following relations:

\begin{align}\label{di di' relation}
d_i^{(0)}=1, \, \,  \, \, \text{and} \, \,  \, \, \sum\limits_{t=0}^rd_i^{(t)}d_i'^{(r-t)}=\delta_{r0};
\end{align}
\begin{align}\label{di dj commu}
[d_i^{(r)},d_j^{(s)}]=0;
\end{align}
\begin{align}\label{Drinfeld generators relation 1}
[d_i^{(r)},e^{(s)}]=(\delta_{i1}-\delta_{i2})\sum\limits_{t=0}^{r-1}d_i^{(t)}e^{(r+s-1-t)};
\end{align}
\begin{align}\label{Drinfeld generators relation 2}
[d_i^{(r)},f^{(s)}]=(\delta_{i2}-\delta_{i1})\sum\limits_{t=0}^{r-1}f^{(r+s-1-t)}d_i^{(t)};
\end{align}
\begin{align}\label{Drinfeld generators relation 3}
[e^{(r)},f^{(s)}]=-\sum\limits_{t=0}^{r+s-1}d_1'^{(t)}d_{2}^{(r+s-1-t)};
\end{align}
\begin{align}\label{Drinfeld generators relation 4}
[e^{(r)},e^{(s)}]=\sum\limits_{t=1}^{s-1}e^{(t)}e^{(r+s-1-t)}-\sum\limits_{t=1}^{r-1}e^{(t)}e^{(r+s-1-t)};
\end{align}
\begin{align}\label{Drinfeld generators relation 5}
[f^{(r)},f^{(s)}]=\sum\limits_{t=1}^{r-1}f^{(t)}f^{(r+s-1-t)}-\sum\limits_{t=1}^{s-1}f^{(t)}f^{(r+s-1-t)}.
\end{align}
\end{Theorem}

For any power series $f(u)\in 1+u^{-1}\kk[[u^{-1}]]$,
it follows from the defining relation \eqref{tiju tklv relation} that there is an automorphism defined via
\begin{align}\label{auto definition}
\mu_f:Y_2\rightarrow Y_2;~t_{i,j}(u)\mapsto f(u)t_{i,j}(u).
\end{align}
Thanks to \eqref{t1122 by d} and \eqref{t1221 by d}, on Drinfeld generators
we have
\begin{align}\label{auto:Drinfeld generators}
\mu_f(d_i(u))=f(u)d_i(u), \mu_f(e_i(u))=e_i(u)~\text{and}~\mu_f(f_i(u))=f_i(u).
\end{align}

The following is the {\it PBW theorem} for $Y_2$.
\begin{Theorem}\cite[Theorem 4.14]{BT18}
Ordered monomias in the elements
\begin{align}\label{pbwbasis}
\{d_i^{(r)};~1\leq i\leq 2, r>0\}\cup\{e^{(r)}, f^{(r)};~r>0\}
\end{align}
taken in any fixed ordering form a basis for $Y_2$.
\end{Theorem}

\subsection{Restricted Yangian $\yp$} \label{S:Restricted Yangian}
We proceed to recall the description of the $p$-central elements of $Y_2$ given in \cite{BT18}.
For $i=1,2$, we define
\begin{align}\label{biu}
b_i(u)=\sum\limits_{r\geq 0}b_i^{(r)}u^{-r}:=d_i(u)d_i(u-1)\cdots d_i(u-p+1).
\end{align}
By \cite[Theorem 5.11(2)]{BT18} the elements in
\begin{align}\label{p-center elements}
\{b_i^{(rp)};~1\leq i\leq 2, \, \,  r > 0\}
\cup\{(e^{(r)})^p,(f^{(r)})^p;~ r>0\}
\end{align}
are algebraically independent, and lie in the center $Z(Y_2)$ of $Y_2$.
The subalgebra they generated is called {\it $p$-center} of $Y_2$ and is denoted by $Z_p(Y_2)$.
In terms of diagonal Drinfeld generators, we define a power series by the rule
\begin{align*}
c(u)=\sum\limits_{r\geq 0}c^{(r)}u^{-r}:=d_1(u)d_2(u-1)\in Y_2[[u^{-1}]].   
\end{align*}
The algebra genrated by the coefficients $\{c^{(r)};~r>0\}$ will be denoted $Z_{\rm HC}(Y_2)$.
We call it the {\it Harish-Chandra center} of $Y_2$.
The center $Z(Y_2)$ is generated by $Z_{\rm HC}(Y_2)$ and $Z_p(Y_2)$ (\cite[Theorem 5.11]{BT18}).
According to \cite[Corollary 5.13]{BT18},
the Yangian $Y_2$ is free as a module over $Z_p(Y_2)$ with basis given by the ordered monomials in the generators in \eqref{pbwbasis} in which no exponent is $p$ or more,
we refer to such monomials as {\it $p$-restricted monomials}.

We let $Z_p(Y_2)_+$ be the ideal of $Z_p(Y_2)$ generated by the elements given in \eqref{p-center elements}. Since the generators \eqref{p-center elements} are algebraically independent it follows that $Z_p(Y_2)$ is a polynomial ring and every element of $\Hom_{\kk\operatorname{-alg}}(Z_p(Y_2), \kk)$ is given by specifying a scalar for each generator \eqref{p-center elements}. The kernels of such homomorphisms are maximal ideals, although if $\kk$ is countable then not every maximal ideal is the kernel of such a map. In any case, $Z_p(Y_2)_+ \unlhd Z_p(Y_2)$ is a maximal ideal.

The {\it restricted Yangian} is defined by $\yp:=Y_2/Y_2Z_p(Y_2)_+$ (see \cite[\S 4.3]{GT21}).
Clearly, the images in $\yp$ of the $p$-restricted monomials in the Drinfeld generators of $Y_2$ form a basis of $\yp$.
When working with $\yp$ we often abuse notation by using the same symbols $d_i^{(r)}, e^{(r)}, f^{(r)}, t_{i,j}^{(r)}$ to refer to the elements
of $Y_2$ and their images in $\yp$.
\begin{Remark}\label{remark1}
In fact, the elements $\{b_i^{(r)};~i=1,2, r>0\}$ also belong to the $p$-center $Z_p(Y_2)$ (see \cite[Lemma 5.7, Theorem 5.8]{BT18}).
Hence $b_i(u)$ is equal to $1$ in $\yp[[u^{-1}]]$.
\end{Remark}

\subsection{Hopf algebra structure}
It is well-known that $Y_2$ is a Hopf algebra (see for instance \cite[Theorem 1.5.1]{Mol07}).
Its comultiplication $\Delta$ and antipode $S$ are given by
\begin{align}\label{hopfstructure of y2}
\Delta(t_{i,j}(u))=\sum\limits_{k=1}^2t_{i,k}(u)\otimes t_{k,j}(u),~\ \ \ S(t_{i,j}(u))=t_{i,j}'(u),
\end{align}
where $t_{i,j}'(u)$ is the $(i,j)$-entry of the matrix $T(u)^{-1}$.
The counit sends $t_{i,j}(u)\mapsto \delta_{i,j}$.
Now, we denote by $I_p:=Y_2Z_p(Y_2)_+$ the ideal defined in Subsection \ref{S:Restricted Yangian}.
In this section, we shall prove that $I_p$ is a Hopf ideal of $Y_2$.
Then the restricted Yangian $\yp$ inherits from $Y_2$ the Hopf algebra structure.

\begin{Lemma}\label{useful equality}
The following identities hold in $Y_2[[u^{-1}]]$.
\begin{itemize}
\setlength{\itemsep}{6pt}
\item[(1)] $t_{1,1}(u-1)t_{1,2}(u)=t_{1,2}(u-1)t_{1,1}(u),$
\item[(2)] $t_{1,1}(u)t_{2,1}(u-1)=t_{2,1}(u)t_{1,1}(u-1),$
\item[(3)] $(k+1)t_{1,2}(u-k)t_{1,1}(u)=kt_{1,1}(u)t_{1,2}(u-k)+t_{1,2}(u)t_{1,1}(u-k)$ for $k\in\ZZ_{\geq 0}$.
\end{itemize}
\end{Lemma}
\begin{proof}
These follow from the defining relation \eqref{tiju tklv relation}. 
For example, to get the third relation, set $(i,j,k,l)=(1,2,1,1)$
and $u:=v-k$ in \eqref{tiju tklv relation}, simplify, then replace $v$ by $u$.
\end{proof}

\begin{Lemma}\label{inductionlemma1}
For any positive integer $n$ and $1\leq i\leq n$, we have that
\begin{eqnarray}\label{inductionlemma-formula}
& & \nonumber \binom{n+1}{i}t_{1,2}(u-n)t_{1,1}(u-i+1)\cdots t_{1,1}(u)\\
& & \hspace{40pt} = \binom{n}{i}t_{1,1}(u-i+1)\cdots t_{1,1}(u)t_{1,2}(u-n)\\\nonumber
& & \hspace{80pt} +\binom{n}{i-1}t_{1,2}(u-i+1)t_{1,1}(u-i+2)\cdots t_{1,1}(u)t_{1,1}(u-n).
\end{eqnarray}
\end{Lemma}
\begin{proof}
We work by induction on $n$.
If $n=1$ Lemma \ref{useful equality}(3) implies that
\[2t_{1,2}(u-1)t_{1,1}(u)=t_{1,1}(u)t_{1,2}(u-1)+t_{1,2}(u)t_{1,1}(u-1),\]
which is clearly true. 
Applying again Lemma \ref{useful equality}(3), we know that
\eqref{inductionlemma-formula} always holds for $i=1$.
So assume that $n\geq 2, i\geq 2$ and the statement is true for $n-1$. It follows that
\begin{eqnarray*}
& &\binom{n}{i-1}t_{1,2}(u-n+1)t_{1,1}(u-i+2)\cdots t_{1,1}(u)\\\nonumber
& & \hspace{40pt}=\binom{n-1}{i-1}t_{1,1}(u-i+2)\cdots t_{1,1}(u)t_{1,2}(u-n+1)\\\nonumber
& & \hspace{80pt}+\binom{n-1}{i-2}t_{1,2}(u-i+2)t_{1,1}(u-i+3)\cdots t_{1,1}(u)t_{1,1}(u-n+1).
\end{eqnarray*}
Replacing $u$ by $u-1$ and multiplying on the both right sides by $(n+1)t_{1,1}(u)$ gives
\begin{eqnarray*}
& & (n+1)\binom{n}{i-1}t_{1,2}(u-n)t_{1,1}(u-i+1)\cdots t_{1,1}(u-1)t_{1,1}(u)\\
& & \hspace{40pt} =(n+1)\binom{n-1}{i-1}t_{1,1}(u-i+1)\cdots t_{1,1}(u-1)t_{1,2}(u-n)t_{1,1}(u)\\
& & \hspace{80pt} =(n+1)\binom{n-1}{i-2}t_{1,2}(u-i+1)t_{1,1}(u-i+2)\cdots t_{1,1}(u-1)t_{1,1}(u-n)t_{1,1}(u).
\end{eqnarray*}
Compute using Lemma \ref{useful equality}:
\begin{eqnarray*}
& & (n+1)\binom{n}{i-1}t_{1,2}(u-n)t_{1,1}(u-i+1)\cdots t_{1,1}(u-1)t_{1,1}(u)\\
& & \hspace{20pt} = n\binom{n-1}{i-1}t_{1,1}(u-i+1)\cdots t_{1,1}(u-1)t_{1,1}(u)t_{1,2}(u-n)\\
 & & \hspace{40pt} + \left(\binom{n-1}{i-1}+(n+1)\binom{n-1}{i-2}\right)t_{1,2}(u-i+1)t_{1,1}(u-i+2)\cdots t_{1,1}(u-1)t_{1,1}(u-n)t_{1,1}(u).
\end{eqnarray*}
Simplifying the above, we obtain \eqref{inductionlemma-formula}.
\end{proof}

\begin{Proposition}\label{prop image of delta}
For any positive integer $n$, we have
\begin{eqnarray}\label{delta mul1}
& & \nonumber \Delta(t_{1,1}(u)t_{1,1}(u-1)\cdots t_{1,1}(u-n+1))\\
& & \hspace{20pt} = \sum\limits_{i=0}^n\binom{n}{i}(t_{1,2}(u-n+1)\cdots t_{1,2}(u-i)t_{1,1}(u-i+1)\cdots t_{1,1}(u)\\\nonumber
& & \hspace{60pt} \otimes t_{1,1}(u)\cdots t_{1,1}(u-i+1)t_{2,1}(u-i)\cdots t_{2,1}(u-n+1)).\\
& & \nonumber \\
\label{delta mul2}
& & \nonumber \Delta(t_{1,2}(u)t_{1,2}(u-1)\cdots t_{1,2}(u-n+1))\\
& & \hspace{20pt} =\sum\limits_{i=0}^n\binom{n}{i}(t_{1,2}(u-n+1)\cdots t_{1,2}(u-i)t_{1,1}(u-i+1)\cdots t_{1,1}(u)\\\nonumber
& & \hspace{60pt} \otimes t_{1,2}(u)\cdots t_{1,2}(u-i+1)t_{2,2}(u-i)\cdots t_{2,2}(u-n+1)).
\end{eqnarray}
\end{Proposition}
\begin{proof}
To establish \eqref{delta mul1}, we proceed by induction on $n$. 
The case $n=1$ immediately follows from \eqref{hopfstructure of y2}.
If $n=2$, then the claim follows from Lemma \ref{useful equality}, by a straightforward computation.
To complete the proof, it suffices to verify the following, for $1\leq i\leq n$:
\begin{eqnarray*}
& & \binom{n+1}{i}(t_{1,2}(u-n)\cdots t_{1,2}(u-i)t_{1,1}(u-i+1)\cdots t_{1,1}(u)\\\nonumber
& & \hspace{20pt} \otimes t_{1,1}(u)\cdots t_{1,1}(u-i+1)t_{2,1}(u-i)\cdots t_{2,1}(u-n))\\
& & \hspace{20pt} =\binom{n}{i}(t_{1,2}(u-n+1)\cdots t_{1,2}(u-i)t_{1,1}(u-i+1)\cdots t_{1,1}(u)t_{1,2}(u-n)\\\nonumber
& & \hspace{40pt} \otimes t_{1,1}(u)\cdots t_{1,1}(u-i+1)t_{2,1}(u-i)\cdots t_{2,1}(u-n+1)t_{2,1}(u-n))\\
& & \hspace{60pt} +\binom{n}{i-1}(t_{1,2}(u-n+1)\cdots t_{1,2}(u-i+1)t_{1,1}(u-i+2)\cdots t_{1,1}(u)t_{1,1}(u-n)\\\nonumber
& & \hspace{80pt} \otimes t_{1,1}(u)\cdots t_{1,1}(u-i+2)t_{2,1}(u-i+1)\cdots t_{2,1}(u-n+1)t_{1,1}(u-n)).\\
\end{eqnarray*}
By Lemma \ref{useful equality}(1), the right sides of the above tensor products are equal, using the fact that the elements $\{t_{1,2}^{(r)};~r>0\}$ commute (cf. \cite[(4.2)]{BT18}, for example). It remains to show that 
\begin{eqnarray*}
 & & \binom{n+1}{i}t_{1,2}(u-n)\cdots t_{1,2}(u-i)t_{1,1}(u-i+1)\cdots t_{1,1}(u)\\\nonumber
& & \quad = \binom{n}{i}t_{1,2}(u-n+1)\cdots t_{1,2}(u-i)t_{1,1}(u-i+1)\cdots t_{1,1}(u)t_{1,2}(u-n)\\\nonumber
& & \quad \quad +\binom{n}{i-1}t_{1,2}(u-n+1)\cdots t_{1,2}(u-i+1)t_{1,1}(u-i+2)\cdots t_{1,1}(u)t_{1,1}(u-n).
\end{eqnarray*}
Now Lemma \ref{inductionlemma1} yields the desired conclusion.
Equation \eqref{delta mul2} can be proved in the same fashion.
\end{proof}

We need another description of the $p$-center of $Y_2$ in
terms of the RTT generators.
Let 
\[
s_{i,j}(u)=\sum\limits_{r\geq 0}s_{i,j}^{(r)}u^{-r}:=t_{i,j}(u)t_{i,j}(u-1)\cdots t_{i,j}(u-p+1)\in Y_2[[u^{-1}]].
\]
According to \cite[Theorem 6.9]{BT18}, 
the $p$-center $Z_p(Y_2)$ is generated by $\{s_{i,j}^{(r)};~1\leq i,j \leq2, r>0\}$.

\begin{Corollary}\label{deltaIp=Ip}
We have 
\[\Delta(s_{i,j}(u))=\sum\limits_{k=1}^2s_{i,k}(u)\otimes s_{k,j}(u).\]
In particular,
$\Delta(I_p)\subseteq I_p\otimes Y_2+Y_2\otimes I_p.$
\end{Corollary}
\begin{proof}
For each permutation $\omega\in S_2$, there is an automorphism $\omega:Y_2\rightarrow Y_2$ sending $t_{i,j}^{(r)}\mapsto t_{\omega(i),\omega(j)}^{(r)}$ (easily seen by inspecting the RTT presentation \eqref{RTT relations}). Moreover, it sends $s_{i,j}(u)$ to $s_{\omega(i),\omega(j)}(u)$. We reduce to computing $\Delta(s_{1,1}(u))$ and $\Delta(s_{1,2}(u))$.
Since we are in characteristic $p$, 
Proposition \ref{prop image of delta} confirms the claim.
\end{proof}

Now we need to determine $S(s_{i,j}(u))$, where $S$ is the antipode of $Y_2$ \eqref{hopfstructure of y2}.
Recall that $t_{i,j}'(u)$ is the $(i,j)$-entry of the matrix $T(u)^{-1}$.
We introduce one more family of elements, we let
\begin{align}\label{def:s'ij}
s'_{i,j}(u)=\sum\limits_{r\geq 0}s_{i,j}'^{(r)}u^{-r}:=t'_{i,j}(u)t_{i,j}'(u-1)\cdots t_{i,j}'(u-p+1)\in Y_2[[u^{-1}]].
\end{align}

\begin{Proposition}\label{Sip=ip}
All of the elements $s_{i,j}'^{(r)}$ belong to the $p$-center $Z_p(Y_2)$.
In particular $S(I_p)\subseteq I_p$.
\end{Proposition}
\begin{proof}
Similar to the proof Corollary \ref{deltaIp=Ip},
using the conjugation automorphism $\omega$ which sends $s_{i,j}(u)$ to $s_{\omega(i),\omega(j)}(u)$,
we reduce to proving that all the coefficients of $s'_{2,2}(u)$ and of $s'_{2,1}(u)$ are central.
Since $S$ is an anti-automorphism (\cite[Proposition 1.3.3]{Mol07}),
\eqref{di dj commu} implies that
\[
S(s_{2,2}(u))=S(t_{2,2}(u-p+1))\cdots S(t_{2,2}(u))=t'_{2,2}(u)t'_{2,2}(u-1)\cdots t'_{2,2}(u-p+1).
\]
We write $\tilde{d}_2(u):=d_2(u)^{-1}$ to ease notation,
and the Gauss decomposition yields $t'_{2,2}(u)=\tilde{d}_2(u)$.
Thus, $S(s_{2,2}(u))=s_{2,2}'(u)=b_2(u)^{-1}$.
Now \cite[Theorem 5.11(2)]{BT18} shows that $s'_{2,2}(u)\in Z_p(Y_2)[[u^{-1}]]$.
Again, the Gauss decomposition implies that $t'_{2,1}(u)=-\tilde{d}_2(u)f(u)$.
For all $m\geq 1$, we claim that
\begin{eqnarray}
\label{delicious sandwiches}
\tilde{d}_2(u)\tilde{d}_2(u+1)\cdots \tilde{d}_2(u+m-1) f(u)=f(u+m)\tilde{d}_2(u)\tilde{d}_2(u+1)\cdots \tilde{d}_2(u+m-1).
\end{eqnarray}
To prove it, we set $v=u+m$ in \cite[(4.52)]{BT18} to deduce that
\[
e(u+m)d_2(u)d_2(u+1)\cdots d_2(u+m-1)=d_2(u)d_2(u+1)\cdots d_2(u+m-1)e(u).
\]
Let $\tau:Y_2\rightarrow Y_2$ be the anti-automorphism defined by $\tau(t_{i,j}^{(r)}):=t_{j,i}^{(r)}$ (see \cite[Section 4.5]{BT18}).
In particular, on Drinfeld generators, $\tau(d_i(u))=d_i(u), \tau(e(u))=f(u)$ and $\tau(f(u))=e(u)$.
Combining this with \eqref{di dj commu}, we obtain
\[
d_2(u)d_2(u+1)\cdots d_2(u+m-1)f(u+m)=f(u)d_2(u)d_2(u+1)\cdots d_2(u+m-1).
\]
This establishes \eqref{delicious sandwiches}. Since the elements $\{t_{1,2}^{(r)};~r>0\}$ commute, we have
\begin{eqnarray*}
S(s_{2,1}(u))&=&t'_{2,1}(u)t'_{2,1}(u-1)\cdots t'_{2,1}(u-p+1)=s'_{2,1}(u)\\
&=&(-1)^p\tilde{d}_2(u)f(u)\tilde{d}_2(u-1)f(u-1)\cdots \tilde{d}_2(u-p+1)f(u-p+1)\\
& = &-\tilde{d}_2(u)\tilde{d}_2(u-1)\cdots \tilde{d}_2(u-p+1)f(u-p+1)^p\\
&=&-b_2(u)^{-1}f(u-p+1)^p,
\end{eqnarray*}
where the third equality uses \eqref{delicious sandwiches}. The foregoing observations in conjunction with \cite[Theorem 5.4]{BT18} yield $s'_{2,1}(u)\in Z_p(Y_2)[[u^{-1}]]$.
\end{proof}

\begin{Theorem}\label{Theorem:yp is hopf}
The ideal $I_p$ is a Hopf ideal of $Y_2$. 
In particular, $\yp$ inherits from $Y_2$ the Hopf algebra structure.
\end{Theorem}
\begin{proof}
 This is a direct consequence of Corollary \ref{deltaIp=Ip} and Propostion \ref{Sip=ip}.
\end{proof}

\subsection{Reduced enveloping algebras}\label{sec:reduced enveloping alg}
Recall that a map $\phi : U \to V$ of $\kk$-vector spaces is called {\it $p$-semilinear} if $\phi(\lambda x + \mu y) = \lambda^p \phi(x) + \mu^p \phi(y)$ for $\lambda, \mu \in \kk$ and $x,y \in U$.

A {\it restricted Lie algebra} is a Lie algebra $\fl$ with a map $\fl\rightarrow \fl$
denoted $x\mapsto x^{[p]}$, such that $x\mapsto x^p-x^{[p]}$ is a $p$-semilinear map $\fl \to Z(\fl)$
to the center of the enveloping algebra $U(\fl)$. The map $x \mapsto x^{[p]}$ is called the $p$-mapping, and the {\it $p$-center} of $U(\fl)$ is defined to be the subalgebra $Z_p(\fl)$ of $Z(\fl)$ generated by $\{x^p-x^{[p]};~x\in\fl\}$.

Given $\chi\in\fl^*$, we define $J_\chi$ to be the ideal of $U(\fl)$ generated by $\{x^p-x^{[p]}-\chi(x)^p;~x\in\fl\}$,
and the {\it reduced enveloping algebra corresponding to $\chi$} to be $U_{\chi}(\fl):=U(\fl)/J_{\chi}$.
For  $\chi=0$ one usually calls $U_0(\fl)$ the {\it restricted enveloping algebra} of $\fl$ (see \cite[\S 2.7]{Jan97}).

Let $N\in\mathbb{Z}_{\geq 1}$ and $\gln$ be the general linear Lie algebra,
which is spanned by the matrix units $\{e_{i,j};~1\leq i,j \leq N\}$.
Then $\gln$ is a restricted Lie algebra with $\GL_N$-equivariant $p$-mapping $\gln\rightarrow\gln;~x\mapsto x^{[p]}$ where $x^{[p]}$ denotes the $p$th matrix power of $x\in\gln$.
In particular, we note that $e_{i,j}^{[p]}=\delta_{ij}e_{i,j}$ for $1\leq i,j\leq N$.
For the case $N=2$,
the {\it evalutation homomorphism}
\begin{align}\label{ev hom}
\ev: t_{i,j}(u)\mapsto\delta_{i,j}+e_{i,j}u^{-1}
\end{align}
defines a surjective homomorphism $Y_2\rightarrow U(\mathfrak{gl}_2)$.
Therefore any representation of Lie algebra $\mathfrak{gl}_2$ can be regarded as a representation of $Y_2$ by pullback,
and any irreducible representation of $\mathfrak{gl}_2$ remains irreducible over $Y_2$.
Moreover, by \cite[Theorem 1.1]{GT21} the homomorphism $\pi$ induces an surjective algebra homomorphism
\begin{align}\label{pinp}
\ev^{[p]}:Y_2^{[p]}\twoheadrightarrow U_0(\mathfrak{gl}_2).
\end{align}
By the same token,
the homomorphism \eqref{pinp} allows us to equip any $U_0(\mathfrak{gl}_2)$-module with a structure of $\yp$-module.

\section{Representations of $\yp$}\label{S:rep of Y2}
In this section, we study the representation of the restricted Yangian $\yp$.
To describe the finite dimensional irreducible modules of $\yp$,
we need to construct analogues of highest weight representations in characteristic $0$. This should be compared to \cite[\textsection 10]{Jan97}.

\subsection{Baby Verma modules}
We first give some notation for the PBW basis of $\yp$.
Let
\[
\inn:=\{(i_1,i_2,\cdots);~i_k\in\ZZ_{\geq 0}~\text{and only
finitely many are nonzero}\}
\]
and $|I|:=\sum_{k\geq 1} i_k$ for $I\in\inn$.
Given $I=(i_1,i_2,\cdots)\in\inn$. Set
\[
f^I:=\prod_{r>0}(f^{(r)})^{i_r}\in\yp
\]
and similarly we define
elements $d_1^I,d_2^I$ and $e^I$.
For $I=(i_1,i_2,\cdots)\in\inn$, note that $e^I$ and $f^I$ are zero if some $i_k\geq p$.
Consider the subset of $\inn$
\[
\ip:=\{(i_1,i_2,\cdots)\in\inn;~0\leq i_k<p\}.
\]
So that
\begin{align}\label{pbw basis of yp2}
\{f^{I_1}d_1^{I_2}d_2^{I_3}e^{I_4};~I_1,I_2,I_3,I_4\in\ip\}
\end{align}
is the PBW basis of $\yp$ (see Section \ref{S:Restricted Yangian}).

Let $Y_2^{[p],-}$ denote the subalgebra of $\yp$ generated by all the $f$'s. We note that 
$Y_2^{[p], -}$ is non-negatively graded with all generators in degree 1, thanks to \eqref{Drinfeld generators relation 5} It follows that there is a unique maximal graded ideal $(Y_2^{[p],-})_+$ generated by the elements $\{f^{(r)};~r>0\}$.

\begin{Lemma}\label{elementnilpotent}
For any $x\in(Y_2^{[p],-})_+$, there exists a positive integer $n$ such that $x^n=0$.
\end{Lemma}
\begin{proof}
We fix an order on the set $\{f^{(r)};~r>0\}$ by setting $f^{(1)}<f^{(2)}<\cdots$.
Given $f^{(r_1)}, f^{(r_2)},\dots,f^{(r_t)}\in \{f^{(r)};~r>0\}$,
we consider the element $$f^{(r_1)}f^{(r_2)}\cdots f^{(r_t)}\in Y_2^{[p],-}.$$
Using the relation \eqref{Drinfeld generators relation 5}, we may write
\begin{align}\label{formula1}
f^{(r_1)}f^{(r_2)}\cdots f^{(r_t)}=\sum\limits_{I\in\mathcal{T}}c_If^I,
\end{align}
where $c_I\in\kk$ and $\mathcal{T}$ is a finite subset of $\inn$.
Since both sides of \eqref{Drinfeld generators relation 5} are homogeneous,
it follows that the $f^I$ in \eqref{formula1} is homogeneous with $|I|=t$ for all $I\in\mathcal{T}$.
We put $k:=\min\{r_1,\dots,r_t\}$ as well as $l:=\max\{r_1,\dots,r_t\}$.
In view of \eqref{Drinfeld generators relation 5},
the Drinfeld generators appearing in \eqref{formula1} are contained in the set $\{f^{(r)};~k\leq r\leq l\}$.

Let $x=\sum_{I\in\mathcal{I}}a_If^I\in (Y_2^{[p],-})_+$ be a non-zero element,
where $\mathcal{I}\subseteq\ip$ is a finite subset and $|I|\neq 0$ for all $I\in\mathcal{I}$.
Assume that the Drinfeld generators in $x$ are $f^{(s_1)},f^{(s_2)},\dots,f^{(s_m)}$ with $1\leq s_1<s_2<\cdots<s_m$.
Put $q:=\min\{|I|;~I\in\mathcal{I}\}$.
Apply again the relation \eqref{Drinfeld generators relation 5} one obtains that  $x^n$ has the form
\begin{align}\label{ast sum}
\sum\limits_{J\in\mathcal{J}}b_Jf^J,
\end{align}
where $\mathcal{J}$ is a finite subset of $\inn$ and $|J|\geq nq$ for all $J\in\mathcal{J}$.
On the other hand,
the foregoing observation implies that the Drinfeld generators in \eqref{ast sum} must be contained in the set $\{f^{(r)};~s_1\leq r\leq s_m\}$.
Given $J=(j_1,j_2,\cdots)\in\mathcal{J}$,
we have the summand
\[b_Jf^J=b_J(f^{(s_1)})^{j_{s_1}}(f^{(s_1+1)})^{j_{s_1+1}}\cdots(f^{(s_m)})^{j_{s_m}}.\]
Consequently,
for large enough $n$, it must be zero, as desired.
\end{proof}

\begin{Definition}
A formal power series $f(u)\in \kk[[u^{-1}]]$ is called {\it restricted}
if
\[f(u)f(u-1)\cdots f(u-p+1)=1,\]
which implies that $f(u) \in 1 + u^{-1} \kk[[u^{-1}]]$. Given two formal series 
\begin{eqnarray*}
& & \lambda_1(u)=1+\lambda_1^{(1)}u^{-1}+\lambda_1^{(2)}u^{-2}+\cdots,\\
& & \lambda_2(u)=1+\lambda_2^{(1)}u^{-1}+\lambda_2^{(2)}u^{-2}+\cdots,
\end{eqnarray*}
we say that the tuple $\lambda(u)=(\lambda_1(u),\lambda_2(u))$ is restricted if both $\lambda_1(u)$ and $\lambda_2(u)$ are restricted.
\end{Definition}

\begin{Lemma}\label{restricted polynomial}
Let $g(u)\in \kk[u^{-1}]$ be a polynomial with the decomposition
\begin{align}\label{gu decomp}
g(u)=(1+\alpha_1u^{-1})\cdots(1+\alpha_ku^{-1}).
\end{align}
Then $g(u)$ is restricted if and only if $\alpha_i\in \FF_p$ for all $i$.
\end{Lemma}
\begin{proof}
Assume $g(u)$ is restricted. Multiplying $u^k$ on both sides of \eqref{gu decomp}, we have
\[u^k g(u)=(u+\alpha_1)\cdots (u+\alpha_k).\]
It follows that
\[
\prod_{i=0}^{p-1}(u-i)^k\prod\limits_{i=0}^{p-1}g(u-i)=\prod\limits_{j=0}^{k}\prod\limits_{i=0}^{p-1}(u+\alpha_j-i).
\]
Note that $\prod_{i=0}^{p-1}g(u-i)=1$ because $g(u)$ is restricted.
Then the assertion follows by comparing the roots of the above equation.
The converse can be proven similarly, and so we omit the argument.
\end{proof}

We define the {\it baby Verma module} $\zpl$ corresponding to $\lambda(u)$ as the quotient of $\yp$
by the left ideal generated by the elements $e^{(r)}$ with $r>0$, and by $d_i^{(r)}-\lambda_i^{(r)}$ with $i=1,2$ and $r>0$.

Given a baby Verma module $\zpl$, we denote by $1_{\lambda(u)}$ the image of the element $1\in\yp$ in the quotient.
Clearly, $\zpl=\yp 1_{\lambda(u)}$.
Owing to \eqref{pbw basis of yp2} ,
there is an isomorphism
\[\zpl\cong Y_2^{[p],-}\otimes_{\kk}\kk1_{\lambda(u)}\]
of vector spaces.

Alternatively, the baby Verma modules can be described in terms of the RTT presentation of $Y_2$,
as follows (cf. \cite[Proposition 3.2.2]{Mol07}). 

\begin{Proposition}\label{another def of zpl}
Let $\lambda(u)=(\lambda_1(u),\lambda_2(u))$ be restricted.
The baby Verma module $\zpl$ equals to the quotient of $\yp$ by the left ideal $J$ which is generated by the elements
$t_{1,2}^{(r)}$ with $r>0$ and by $t_{i,i}^{(r)}-\lambda_i^{(r)}$ with $i=1,2$ and $r>0$.
\end{Proposition}
\begin{proof}
Let $I$ be the left ideal generated by the elements $e^{(r)}$ with $r>0$,
and by $d_i^{(r)}-\lambda_i^{(r)}$ with $i=1,2$ and $r>0$.
It suffices to show that $I=J$.

Since $e(u)\equiv 0~(\Mod I)$, it follows from \eqref{t1221 by d} that
\[t_{12}(u)=d_1(u)e(u)\equiv 0~(\Mod I),\]
so that $t_{1,2}^{(r)}\in I$.
Note that $t_{1,1}(u)=d_1(u)$, so that $t_{1,1}^{(r)}-\lambda_1^{(r)}\in I$.
Using \eqref{t1122 by d} one obtains
\[t_{2,2}(u)-\lambda_2(u)=f(u)d_1(u)e(u)+d_2(u)-\lambda_2(u)\equiv d_2(u)-\lambda_2(u)\equiv 0~(\Mod I).\]
Consequently, $J\subseteq I$. One argues similarly for $I\subseteq J$ applying again \eqref{t1122 by d} and \eqref{t1221 by d}.
\end{proof}

\begin{Proposition}\label{zpl unique maximal submodule}
$\zpl$ has a unique maximal submodule.
\end{Proposition}
\begin{proof}
We shall prove any proper submodule $M$ of $\zpl$ is contained in $(Y_2^{[p],-})_+.1_{\lambda(u)}$.
Suppose that $M\nsubseteq (Y_2^{[p],-})_+.1_{\lambda(u)}$,
then there is a nonzero element $y=(1-x).1_{\lambda(u)}\in M$ and $x\in (Y_2^{[p],-})_+$.
Lemma \ref{elementnilpotent} provides an integer $n$ such that $x^n=0$.
We have for $x':=1+x+\cdots+x^{n-1}$, $x'y=1_{\lambda(u)}\in M$,
so that $M=\zpl$, a contradiction.
Obviously,  $(Y_2^{[p],-})_+.1_{\lambda(u)}$ is a proper subspace of $\zpl$.
Hence, the sum of all proper submodules of $\zpl$ is the unique maximal submodule.
\end{proof}

\begin{Definition}
Given a restricted tuple $\lambda(u)=(\lambda_1(u),\lambda_2(u))$,
the irreducible representation $L^{[p]}(\lambda(u))$ of $\yp$ is defined as the quotient of the baby Verma module $\zpl$ by the unique maximal proper submodule.
\end{Definition}

Let $Y_2^{[p],0}$ denote the subalgebra of $\yp$ generated by all $d'$s.
The one dimensional quotient $L^{[p]}(\lambda(u))/(Y_2^{[p],-})_+L^{[p]}(\lambda(u))$
can be viewed as a $Y_2^{[p],0}$-module and $d_i(u)$ acts on this space via $\lambda_i(u)$.
So we get $L^{[p]}(\lambda(u))\cong L^{[p]}(\nu(u))$ if and only if $\lambda(u)=\nu(u)$.

Let $L$ be a finite dimensional representation of $\yp$.
We define the subspace $L^0$ of $L$ via
\[L^0:=\{v\in L;~e(u)v=0\}.\]

\begin{Lemma}\label{l0 nonzero}
If $L$ is a finite dimensional representation of $\yp$,
then the space $L^0$ is nonzero.
\end{Lemma}
\begin{proof}
Suppose that $L^0=(0)$.
Since $e^{(1)}$ is nilpotent,
there exists a nonzero vector $v_1 \in L$ such that $e^{(1)}v_1=0$.
By assumption, we can find an positive integer $n_1$ such that
$e^{(r)}v_1 = 0$ for all $1\leq r\leq n_1$ and $e^{(n_1+1)}v_1\neq 0$.

Again, the element $e^{(n_1+1)}$ is nilpotent.
There exists some integer $t$ with $1\leq t\leq p-1$ such that $(e^{(n_1+1)})^lv_1\neq 0$ for all $1\leq l\leq t$
and $(e^{(n_1+1)})^{t+1}v_1=0$.
Setting $v_2:=(e^{(n_1+1)})^tv_1\neq 0$,
we obtain $e^{(n_1+1)}v_2=0$.
For $1\leq i\leq n_1$,
the relation \eqref{Drinfeld generators relation 4} yields
\[e^{(i)}v_2=e^{(i)}(e^{(n_1+1)})^tv_1=\sum c_{a_1,\cdots,a_k}(e^{(n_1+1)})^{a_1}(e^{(n_1)})^{a_2}\cdots (e^{(i)})^{a_k}v_1,\]
where $k=n_1+2-i$ and $a_1+a_2+\cdots +a_k=t+1$.
The choice of $v_2$ implies that the above sum is zero.
This show that $e^{(r)}v_2=0$ for all $1\leq r\leq n_1+1$.
So that we can find an positive integer $n_2>n_1$ such that
$e^{(r)}v_2=0$ for all $1\leq r\leq n_2$ and $e^{(n_2+1)}v_2\neq 0$.

Repeat the above argument, we obtain two sequences 
$v_1,v_2,\dots$ and 
$1\leq n_1<n_2<\cdots$
which satisfy 
(a) $v_i\neq 0$;
(b) $e^{(1)}v_i=e^{(2)}v_i=\cdots=e^{(n_i)}v_i=0$;
(c) $e^{(n_i+1)}v_i\neq 0$ for every positive integer $i$.
Let $m$ be an arbitrary positive integer.
Assume that $c_1v_1+c_2v_2+\cdots +c_mv_m=0$.
It follows from (b) that $c_1e^{(n_1+1)}v_1=0$,
and property (c) gives $c_1=0$.
By the same token, we obtain all $c_i$ are zero.
Therefore $v_1,v_2,\cdots,v_m$ are linearly independent.
As $L$ is finite dimensional,
we arrive at a contradiction.
\end{proof}

\begin{Theorem}\label{fd irreducible rep}
Every finite dimensional simple $\yp$-module is isomorphic to some $L^{[p]}(\lambda(u))$.
\end{Theorem}
\begin{proof}
This is very similar to the proof in characteristic zero explained in \cite[Theorem 3.2.7]{Mol07}.
We just give a brief account by using the Drinfeld generators.

By Lemma \ref{l0 nonzero},  we know that $L^0\neq (0)$.
We show that the subspace $L^0$ is invariant with respect the action of all elements $d_i^{(r)}$.
If $v\in L^0$, then \eqref{Drinfeld generators relation 1} implies that $[d_i^{(r)},e^{(s)}]v=0$ for all $r,s>0$.
It follows that $e^{(s)}d_i^{(r)}v=0$, so that $d_i^{(r)}v\in L^0$.

Furthermore, \eqref{di dj commu} implies that
the elements $d_i^{(r)}$ with $i=1,2$ and $r>0$ act on $L^0$ as pairwise commuting operators. 
Hence, there exists a nonzero vector $\zeta\in L^0$ such that $d_i^{(r)}\zeta=\lambda_i^{(r)}\zeta$,
where $\lambda_i^{(r)}\in\kk$.
Letting $\lambda_i(u):=1+\lambda_i^{(1)}u^{-1}+\lambda_i^{(2)}u^{-1}+\cdots\in\kk[[u^{-1}]]$,
we put $\lambda(u)=(\lambda_1(u),\lambda_2(u))$.
Then $d_i(u)\zeta=\lambda_i(u)\zeta$.
Since $1=b_i(u)=d_i(u)d_i(u-1)\cdots d_i(u-p+1)$ in $\yp$ (Remark \ref{remark1}),
it follows that $\lambda(u)$ is restricted.
Note that $L$ is irreducible and $L=\yp\zeta$.
Clearly, there is a surjective homomorphism $\zpl\twoheadrightarrow L;~1_{\lambda(u)}\mapsto\zeta$.
Proposition \ref{zpl unique maximal submodule} now yields an isomorphism $L\cong L^{[p]}(\lambda(u))$.
\end{proof}

\subsection{Evaluation modules}\label{S:evaluation module}
For any $a,b\in\FF_p$,
we consider the irreducible $U_0(\mathfrak{gl}_2)$-module $L(a,b)$.
The module $L(a,b)$ is generated by a vector $\xi$ and 
\begin{align}\label{maximal vector}
e_{1,1}\xi=a \xi, \, \, \, \, e_{2,2}\xi=b\xi \, \, \, \,  \text{and} \, \, \, \, e_{1,2}\xi=0.
\end{align}
For $n \in\FF_p$, we denote by $[n]\in\NN$ the minimal element such that $[n]\equiv n~(\Mod p)$.
The module $L(a,b)$ has a basis $e_{2,1}^k\xi$ for $k$ from $0$ to $[a-b]$ (cf. \cite[\S 5.2]{Jan97}).
The homomorphism \eqref{pinp} allows us to view $L(a,b)$ as a $\yp$-module.
Namely, for any indices $i,j\in\{1,2\}$ the generator $t_{i,j}^{(1)}$ acts on $L(a,b)$ as $e_{i,j}$,
while any generator $t_{i,j}^{(r)}$ with $r\geq 2$ acts as the zero operator.
We will keep the same notation $L(a,b)$ for this $\yp$-module and call it the {\it evaluation module}.
By \eqref{another def of zpl}, 
it is clear that $L(a,b)$ is isomorphic to the module $L^{[p]}(\lambda_1(u),\lambda_2(u))$,
where $\lambda_1(u)=1+a u^{-1}$ and $\lambda_2(u)=1+b u^{-1}$.

Let $L$ and $M$ be two $Y_2$-modules.
Since the algebra $Y_2$ is a Hopf algebra,
the tensor product space $L\otimes M$ can be equipped with a $Y_2$-action with the use of the comultiplication $\Delta$ on $Y_2$ (see \eqref{hopfstructure of y2}) by the rule
 \[x.(l\otimes m):=\Delta(x)(l\otimes m),~x\in Y_2,~l\in L,~m\in M.\]
Furthermore, since $\yp$ is a Hopf quotient of $Y_2$ (Theorem \ref{Theorem:yp is hopf}), this comultiplication descends to $\yp$, and we will use the same symbol $\Delta$.
By the same token, if both $L$ and $M$ are $\yp$-modules, then $L\otimes M$ can be equipped with a $\yp$-module structure with the use of the comultiplication $\Delta$.
 
\begin{Definition}
Given two sequences $\alpha = (\alpha_1,...,\alpha_k)$ and $\beta = (\beta_1,...,\beta_k)$ of elements of $\FF_p$, we say that $(\alpha, \beta)$ is {\it non-decreasing}
if for all $i=1,...,k$ the integer $[\alpha_i-\beta_i]$ is minimal in $$\{[\alpha_j-\beta_l] : i\leq j,l\}.$$
\end{Definition}

Write
\[
\lambda_1(u):=(1+\alpha_1u^{-1})\cdots(1+\alpha_ku^{-1})\]
and 
\[ \lambda_2(u):=(1+\beta_1u^{-1})\cdots(1+\beta_ku^{-1}).\vspace{4pt}\]
Both $\lambda_1(u)$ and $\lambda_2(u)$ are restricted by Lemma \ref{restricted polynomial} because all the elements $\alpha_i$ and $\beta_i$ belong to $\FF_p$.

\begin{Proposition}\label{tensor product irreducible}
Given non-decreasing sequences $(\alpha, \beta)$ of elements of $\FF_p$, the tensor product
\begin{align}\label{tensor product}
L(\alpha_1,\beta_1)\otimes\cdots\otimes L(\alpha_k,\beta_k)
\end{align}
is an irreducible $\yp$-module and is isomorphic to $L^{[p]}(\lambda(u))=L^{[p]}(\lambda_1(u),\lambda_2(u))$.
\end{Proposition}
\begin{proof}
We can use the same proof \cite[Proposition 3.3.2]{Mol07} (see also \cite[Proposition 3.2.3]{Kal20} ).

Denote the module \eqref{tensor product} by $L$.
Let $\xi_i$ be the generating vector of $L(\alpha_i,\beta_i)$ \eqref{maximal vector} for $i=1,\dots,k$.
Using the definition \eqref{hopfstructure of y2} of $\Delta$, we obtain $t_{1,2}(u)\xi=0$, where $\xi=\xi_1\otimes\cdots\xi_k$.
It suffices to show that any vector $\zeta\in L$ satisfying $t_{1,2}(u)\zeta=0$ is proportional to $\xi$.
Now proceed by induction on $k$.
Write any such vector $\zeta=\sum_{r=0}^s e_{1,2}^r\xi_1\otimes \zeta_r$ and $\zeta_r$ some elements of $L(\alpha_2,\beta_2)\otimes\cdots\otimes L(\alpha_k,\beta_k)$.
We may assume that $s\leq [\alpha_1-\beta_1]$.
Then repeating all the steps of the proof of \cite[Proposition 3.3.2]{Mol07}, we obtain
the following relation:
\[
s(\alpha_1-\beta_1-s+1)(\alpha_1-\beta_2-s+1)\dots(\alpha_1-\beta_k-s+1)=0 .
\]
Since $0\leq [\alpha_1-\beta_1]\leq p-1$ and $0\leq s\leq [\alpha_1-\beta_1]$,
it follows that $(\alpha_1-\beta_1-s+1)=0$ only if $s=0$ and $[\alpha_1-\beta_1]=p-1$.
Our assumption on the parameter $\alpha_i$ and $\beta_i$ now implies that $[\alpha_1-\beta_j]\geq [\alpha_1-\beta_1]$ for all other $j$.
So that  $(\alpha_1-\beta_j-s+1)=0$ can be zero again only for $s=0$ and the claim follows.
\end{proof}

We continue to denote by $L$ the tensor product \eqref{tensor product}.
\begin{Proposition}\label{Prop:permutation is still irrdecible}
Suppose that the representation $L$ is irreducible.
Then any permutation of the tensor factors in $L$ gives an isomorphic representation of $\yp$.
\end{Proposition}
\begin{proof}
Let $L'$ be the representation obtained from $L$ by some permutation of the tensor factors,
then $L'$ has the form $L'=L(\alpha_{i_1},\beta_{i_1})\otimes\cdots\otimes L(\alpha_{i_k},\beta_{i_k})$,
where $\{i_1,\cdots,i_k\}=\{1,\cdots, k\}$.
Let $\xi_i$ be the generating vector of $L(\alpha_i,\beta_i)$ \eqref{maximal vector}.
Denote by $\eta:=\xi_{i_1}\otimes\cdots\otimes \xi_{i_k}$.
Using the definition \eqref{hopfstructure of y2} of $\Delta$ and \eqref{maximal vector},
we know that the submodule $\yp.\eta$ of $L'$ generated by $\eta$ is a homomorphic image of $Z^{[p]}(\lambda_1(u),\lambda_2(u))$.
As $L$ is irreducible, Proposition \ref{tensor product irreducible} implies that $L\cong L^{[p]}(\lambda_1(u),\lambda_2(u))$.
Therefore, $L$ is isomorphic to a subquotient of $L'$.
Since $L$ and $L'$ have the same dimension, 
they are isomorphic.
\end{proof}

From Theorem \ref{fd irreducible rep}, we know that every finite dimensional irreducible module of $\yp$ is isomorphic to a one of the simple modules $L^{[p]}(\lambda(u)) = L^{[p]}(\lambda_1(u),\lambda_2(u))$.
We first consider the case that both $\lambda_1(u)$ and $\lambda_2(u)$ are polynomials in $u^{-1}$.

\begin{Theorem}\label{two poly are tensor}
Suppose that $\lambda_1(u)$ and $\lambda_2(u)$ are restricted polynomials.
Then $L^{[p]}(\lambda(u)) = L^{[p]}(\lambda_1(u),\lambda_2(u))$ is isomorphic to some tensor product of evaluation modules.
In particular, $L^{[p]}(\lambda_1(u),\lambda_2(u))$ has finite dimension.
\end{Theorem}
\begin{proof}
Write the decomposition $\lambda_1(u)=(1+\alpha_1u^{-1})\cdots(1+\alpha_ku^{-1})$
and $ \lambda_2(u)=(1+\beta_1u^{-1})\cdots(1+\beta_ku^{-1})$ for some $k\geq 0$ and some parameters 
$\alpha_i, \beta_i$.
By Lemma \ref{restricted polynomial}, all of the elements $\alpha_i, \beta_i$ belong to the field $\FF_p$.
Reorder them so that $(\alpha, \beta)$ is non-decreasing in accordance with Proposition \ref{tensor product irreducible}. We obtain the isomorphism
\[ 
L^{[p]}(\lambda_1(u),\lambda_2(u))\cong L(\alpha_1,\beta_1)\otimes\cdots\otimes L(\alpha_k,\beta_k),
\]
as required.
\end{proof}

\begin{Lemma}\label{exist f}
If $L^{[p]}(\lambda(u)) = L^{[p]}(\lambda_1(u),\lambda_2(u))$ is finite dimensional
then there is a formal series $f(u)\in 1+u^{-1}\kk[[u^{-1}]]$ such that
$f(u)\lambda_1(u)$ and $f(u)\lambda_2(u)$ are polynomials in $u^{-1}$.
\end{Lemma}
\begin{proof}
The proof of
\cite[Proposition 3.3.1]{Mol07} works verbatim in our setting.
\end{proof}

\begin{Remark}\label{rk2}
It should be noted that the $f(u)\lambda_1(u)$ and $f(u)\lambda_2(u)$ in the above lemma might no longer be restricted.
Given a formal power series $f(u)\in 1+u^{-1}\kk[[u^{-1}]]$, we recall the automorphism $\mu_f$ \eqref{auto definition} of $Y_2$.
If $\mu_f$ factors to an automorphism of $\yp$, then clearly $f$ must be restricted.
For general restricted $\lambda(u)=(\lambda_1(u),\lambda_2(u))$, 
there may not exist a restricted power series $f(u)$ such $f(u)\lambda_1(u)$ and $f(u)\lambda_2(u)$ are polynomials even $L^{[p]}(\lambda(u))$ is finite dimensional (see Remark \ref{ex1}).
\end{Remark}

\begin{Theorem}\label{Drinfeld poly}
The irreducible representation $L^{[p]}(\lambda(u)) = L^{[p]}(\lambda_1(u),\lambda_2(u))$ is finite dimensional if and only if there exists a monic polynomial $P(u)$ in $u$ such that 
\begin{align}\label{drinfeld poly}
\frac{\lambda_1(u)}{\lambda_2(u)}=\frac{P(u+1)}{P(u)}.
\end{align}
\end{Theorem}
\begin{proof}
Suppose that the representation $L^{[p]}(\lambda_1(u),\lambda_2(u))$ is finite dimensional.
Then by Lemma \ref{exist f} we can find a formal series $f(u)$ such that
$f(u)\lambda_1(u)=(1+\alpha_1u^{-1})\cdots(1+\alpha_ku^{-1})$
and $f(u) \lambda_2(u)=(1+\beta_1u^{-1})\cdots(1+\beta_ku^{-1})$ for some $k\geq 0$ and some scalars
$\alpha_i, \beta_i$.
Note that $\lambda_1(u)/\lambda_2(u)=f(u) \lambda_1(u)/f(u) \lambda_2(u)$ is restricted,
so that for each $\alpha_i$, there exists $\beta_j$ such that $\alpha_i-\beta_j\in\FF_p$.
Reordering them if necessary,
we may thus assume that  $\alpha_i-\beta_i\in\FF_p$.
The the polynomial 
\[
P(u)=\prod\limits_{i=1}^{k}(u+\beta_i)(u+\beta_i+1)\cdots(u+\alpha_i-1)
\]
satisfies \eqref{drinfeld poly}, which can be checked by a short computation.

Conversely, suppose \eqref{drinfeld poly} holds for a polynomial $P(u)=(u+\gamma_1)\cdots(u+\gamma_s)$.
Set 
\begin{eqnarray*}
\label{sausages}
\begin{array}{l}
 \mu_1(u): =(1+(\gamma_1+1)u^{-1})\cdots (1+(\gamma_s+1)u^{-1}),\\
\mu_2(u): =(1+\gamma_1 u^{-1})\cdots (1+\gamma_s u^{-1}).
\end{array}
\end{eqnarray*}
By construction we have
\begin{eqnarray}
\label{eggs}
\frac{\mu_1(u)}{\mu_2(u)} = \frac{P(u+1)}{P(u)}.
\end{eqnarray}
For each $\gamma_i$, we consider the $\mathfrak{gl}_2$-module $L(\gamma_i+1,\gamma_i)$ which is generated by $\xi_i$ and the module structure is given by
\begin{align}\label{2d module}
e_{1,1}\xi_i=(\gamma_i+1)\xi_i, e_{2,2}\xi_i=\gamma_i\xi_i, e_{1,2}\xi_i=0, e_{2,1}^2\xi_i=0.
\end{align}
In particular,  the module $L(\gamma_i+1,\gamma_i)$ is two dimensional with basis $\{\xi_i, e_{2,1}\xi_i\}$.
We can regard them as $Y_2$-modules via \eqref{ev hom} and consider the tensor product module
\[
L:=L(\gamma_1+1,\gamma_1)\otimes\cdots\otimes L(\gamma_s+1,\gamma_s).
\]
Clearly, $\dim_{\kk}L=2^s$.
We put $\xi:=\xi_1\otimes\cdots\otimes \xi_s$.
Using the comultiplication \eqref{hopfstructure of y2} in conjuntion with \eqref{2d module} one can show by direct computation that
\[
t_{1,2}(u)\xi=0, \, \, \, \, t_{1,1}(u)\xi=\mu_1(u)\xi, \, \, \, \, t_{2,2}(u)\xi=\mu_2(u)\xi
\]
and $t_{2,1}(u)$ acts nilpotently on the module $L$.  Let $M:=Y_2\xi\subseteq L$ be the cyclic submodule of $L$.
By twisting the action of $Y_2$ on $M$ by the automorphism \eqref{auto definition} with $f(u)=\mu_2(u)^{-1}$,
we obtain a $Y_2$-module which will be denoted $\widetilde{M}$.
Clearly, $\widetilde{M}$ is also generated by $\xi$ and, thanks to \eqref{drinfeld poly} and \eqref{eggs}, we have
\begin{eqnarray}
\label{urghpomple}
t_{1,2}(u)\xi=0, \, \, \, \,  t_{1,1}(u)\xi=\frac{\mu_1(u)}{\mu_2(u)}\xi=\frac{\lambda_1(u)}{\lambda_2(u)}\xi,\, \, \, \,  t_{2,2}(u)\xi=\xi.
\end{eqnarray}

We define a subspace $V$ of $\widetilde{M}$
\[
V:={\rm span}\{s_{2,1}^{(r_1)}s_{2,1}^{(r_2)}\cdots s_{2,1}^{(r_m)}\xi;~m\geq 0, r_1,r_2,\dots,r_m\geq 1\}.
\]
The set $\{s_{2,1}^{(r)};~r\geq 1\}$ is a Lie subset that acts on $V$ by nilpotent transformations.
Therefore the Engel-Jacobson theorem provides a vector $\eta\in V$ such that $s_{2,1}(u)\eta=0$.
Then we consider the submodule $Y_2\eta$ of the module $\widetilde{M}$.
By \eqref{urghpomple}, we obtain
\[
t_{1,2}(u)\eta=0, \, \, \, \, t_{1,1}(u)\eta=\frac{\lambda_1(u)}{\lambda_2(u)}\eta, \, \, \, \, t_{2,2}(u)\eta=\eta.
\]
There results a surjective homomorphism
\[Z^{[p]}(\frac{\lambda_1(u)}{\lambda_2(u)},1)\twoheadrightarrow Y_2\eta;~1_{\lambda_1(u)/\lambda_2(u)}\mapsto \eta\]
where $1_{\lambda_1(u)/\lambda_2(u)}$ is the cyclic generator. As a consequence, $L^{[p]}(\frac{\lambda_1(u)}{\lambda_2(u)},1)$ is finite dimensional.
Since $\lambda_2(u)$ is restricted,
we apply again the twisted action of $Y_2^{[p]}$  on $L^{[p]}(\frac{\lambda_1(u)}{\lambda_2(u)},1)$ with $f(u)=\lambda_2(u)$ (see Remark \ref{rk2}) to obtain the module $L^{[p]}(\lambda_1(u),\lambda_2(u))$.
We conclude that the module $L^{[p]}(\lambda_1(u),\lambda_2(u))$ is also finite dimensional.
\end{proof} 

\begin{Remark}
As observed by Kalinov ( \cite [Remark, p. 6985]{Kal20}), 
the  polynomial $P(u)$ in Theorem \ref{Drinfeld poly} is not unique.
Suppose that $Q(u)$ is another monic polynomial in $u$ and $\frac{P(u+1)}{P(u)}=\frac{Q(u+1)}{Q(u)}$.
It follows that $\frac{P(u)}{Q(u)}=: F(u)$ satisfies $F(u+1)=F(u)$.
Thus $F(u)$ is a ratio of products of expressions of the form $(u+\alpha)^p-(u+\alpha)$ for some $\alpha\in\kk$.
\end{Remark}

\begin{Remark}\label{ex1}
Let $\alpha\notin\FF_p$.
We consider the $2$-dimensional $\mathfrak{gl}_2$-module $L(\alpha+1,\alpha)$ as defined in \eqref{2d module}.
By twisting the action of $Y_2$ on  $L(\alpha+1,\alpha)$ by the automorphism \eqref{auto definition} with $f(u)=(1+\alpha u^{-1})^{-1}$, we obtain a module over $\yp$ which is isomorphic to the irreducible module
$L^{[p]}(\frac{1+(\alpha+1)u^{-1}}{1+\alpha u^{-1}},1)$.
However, there does not exist a restricted polynomial $g(u)$ in $u^{-1}$ such that $g(u)\frac{1+(\alpha+1)u^{-1}}{1+\alpha u^{-1}}$ is a restricted polynomial.
\end{Remark}

Denote by $L$ the tensor product \eqref{tensor product}.
We will now establish a criterion of irreducibility of the representations $L$.
For $\alpha,\beta\in\FF_p$, we define the {\em string} corresponding to the pair $(\alpha,\beta)$ as the set
\[
S(\alpha,\beta):=\left\{
\begin{array}{ll}
\{[\beta],[\beta]+1,\dots,[\alpha]-1\}&~\ \ \ \ \text{if}~~[\alpha]>[\beta],\\
\{[\beta],[\beta]+1,\dots,[\alpha]+p-1\}&~\ \ \ \ \text{if}~~[\alpha]<[\beta],\\
\emptyset &~\ \ \ \ \text{if}~[\alpha]=[\beta].
\end{array}
\right.\]
Note that we always view $S(\alpha,\beta)$ as a subset of $\FF_p:=\{\overline{0},\cdots,\overline{p-1}\}$.
By definition, we have $|S(\alpha,\beta)|=[\alpha-\beta]$, where $|S(\alpha,\beta)|$ is the cardinality of the set $S(\alpha,\beta)$.
\begin{Example}
Suppose that $p=7$.
If $(\alpha,\beta)=(1,3)$, then $S(\alpha,\beta)=\{\overline{3},\overline{4},\overline{5},\overline{6},\overline{0}\}$. 
\end{Example}
\begin{Definition}
Two strings $S_1$ and $S_2$ are {\em in general position} if either
\begin{itemize}
    \item[(i)] $S_1\cup S_2$ is not a string, or\medskip 
    \item[(ii)] $S_1\subseteq S_2$ or $S_2\subseteq S_1$.
\end{itemize}
\end{Definition}

The following result is adapted from \cite[Corollary 3.3.6]{Mol07}.
Since our definition about string is slightly different from that given in \cite[Page 115]{Mol07},
we provide a proof here.
\begin{Proposition}
\label{classify irreducible tensor prods}
The $\yp$-module $L$ is irreducible if and only if the strings $S(\alpha_1,\beta_1),\dots,S(\alpha_k,\beta_k)$ are pairwise in general position.
\end{Proposition}
\begin{proof}
Suppose that the strings are pairwise in general position and assume first that $[\alpha_1-\beta_1]\leq\cdots\leq [\alpha_k-\beta_k]$. Let $\alpha:=(\alpha_1,\dots,\alpha_k)$ and $\beta:=(\beta_1,\dots,\beta_k)$. One easily verifies that $(\alpha,\beta)$ is non-decreasing.
For example, let us prove that $[\alpha_1-\beta_1]\leq [\alpha_i-\beta_j]$ for $i,j\geq 1$.
We may assume that $i<j$.
As $[\alpha_1-\beta_1]\leq [\alpha_i-\beta_i]$,
it suffices to prove that $[\alpha_i-\beta_i]\leq [\alpha_i-\beta_j]$.
Since $S_i:=S(\alpha_i,\beta_i)$ and $S_j:=S(\alpha_j,\beta_j)$ are in general position
and $[\alpha_i-\beta_i]\leq [\alpha_j-\beta_j]$,
either $S_i\subseteq S_j$ or $S_i \cup S_j$ is not a string.
In either case, 
we always have $S(\alpha_i,\beta_j)\supseteq S_i$.
It follows that $[\alpha_i-\beta_j]=|S(\alpha_i,\beta_j)|\geq |S_i|=[\alpha_i-\beta_i]$.
Then Proposition \ref{tensor product irreducible} implies $L$ is irreducible.
Any permutation of the tensor factors yields an isomorphic representation by Proposition \ref{Prop:permutation is still irrdecible}.

Conversely, let $k=2$ and let $L(\alpha_1,\beta_1)\otimes L(\alpha_2,\beta_2)$ is irreducible.
Suppose that the strings $S(\alpha_1,\beta_1)$ and $S(\alpha_2,\beta_2)$ are not in general position.
Then the strings $S(\alpha_1,\beta_2)$ and $S(\alpha_2,\beta_1)$ are in general position.
Hence the representation $L(\alpha_1,\beta_2)\otimes L(\alpha_2,\beta_1)$ of $\yp$ is irreducible.
Note that both $L(\alpha_1,\beta_2)\otimes L(\alpha_2,\beta_1)$ and $L(\alpha_1,\beta_1)\otimes L(\alpha_2,\beta_2)$ are isomorphic to 
$L^{[p]}(\lambda_1(u),\lambda_2(u))$,
where $\lambda_1(u)=(1+\alpha_1u^{-1})(1+\alpha_2u^{-1})$ and $\lambda_2(u)=(1+\beta_1u^{-1})(1+\beta_2u^{-1})$.
Hence they have the same dimension:
\[
([\alpha_1-\beta_1]+1)([\alpha_2-\beta_2]+1)=([\alpha_1-\beta_2]+1)([\alpha_2-\beta_1]+1).
\]
It follows that 
\[
((\alpha_1-\beta_1)+1)((\alpha_2-\beta_2)+1)=((\alpha_1-\beta_2)+1)((\alpha_2-\beta_1)+1).
\]
This implies that $(\alpha_1-\alpha_2)(\beta_1-\beta_2)=0$,
and hence that $S(\alpha_1,\beta_1)$ and $S(\alpha_2,\beta_2)$ are in general position.
We arrive at a contradiction. For the general case.
Suppose that $L$ is irreducible, but a pair of strings $S(\alpha_i,\beta_i)$ and $S(\alpha_j,\beta_j)$ are not in general position.
Applying Proposition \ref{Prop:permutation is still irrdecible} and permuting the tensor factors of $L$ if necessaary, we may assume that $i$ and $j$ are adjacent.
However, the representation $L(\alpha_i,\beta_i)\otimes L(\alpha_j,\beta_j)$ is reducible as shown above.
This implies that $L$ is reducible, a contradiction.
\end{proof}

\section{Finite $W$-algebras}\label{Section:W-algebras}
Now we turn to the applications of our earlier results, which go via the theory of finite $W$-algebras.
We fix $n\in\ZZ_{\geq 1}$.
Let $G:=GL_{2n}(\kk)$ be the general linear group of degree $2n$ with Lie algebra $\fg:=\Lie(G)=\mathfrak{gl}_{2n}$.
We write $\{e_{i,j};~1\leq i,j\leq 2n\}$ for the standard basis of $\fg$ consisting of matrix units.
In this section, we only consider the $W$-algebras associated to $2\times n$-rectangular nilpotent elements in $\fg$.
The reader is referred to \cite{GT19-1} for the theory of modular finite $W$-algebras.

\subsection{Restricted finite $W$-algebras}\label{Restricted finite W-algebras}
We consider the partition $(n,n)\vdash 2n$ of $2n$. 
Let $\pi$ be the corresponding pyramid.
The boxes in the pyramid are numbered along rows from left to right and from top to bottom.
For example, if $n=5$, then the pyramid associated to $(5,5)$ is 
\begin{equation*}
\begin{array}{c}
\begin{picture}(65,26)
 \put(0,0){\line(1,0){65}}
\put(0,13){\line(1,0){65}} 
\put(0,26){\line(1,0){65}}
\put(0,0){\line(0,1){26}} 
\put(13,0){\line(0,1){26}}
\put(26,0){\line(0,1){26}} 
\put(39,0){\line(0,1){26}}
\put(52,0){\line(0,1){26}}
\put(65,0){\line(0,1){26}}
\put(3,14.5){\hbox{1}}
\put(16,14.5){\hbox{2}}
\put(29,14.5){\hbox{3}}
\put(42,14.5){\hbox{4}}
\put(55,14.5){\hbox{5}}
\put(3,2.5){\hbox{6}}
\put(16,2.5){\hbox{7}}
\put(29,2.5){\hbox{8}}
\put(42,2.5){\hbox{9}}
\put(52,2.5){\hbox{10}}
\end{picture}
\end{array}.
\end{equation*}
For $1\leq i\leq n$, we use the notation $i':=i+n$.
The box in $\pi$ containing $i$ is referred to as the $i$th box,
and let $\row(i)$ and $\col(i)$ denote the {\it row} and {\it column} numbers of the brick in which $i$ appears, respectively.
We therefore have $\row(i)=1, \row(i')=2$ and $\col(i)=\col(i')=i$ for every $1\leq i\leq n$.

The pyramid $\pi$ is used to determine the nilpotent element
\begin{align}\label{def e}
e:=\sum\limits_{1\leq i\leq n-1}e_{i,i+1}+\sum\limits_{1\leq i\leq n-1}e_{i',(i+1)'}\in \fg,
\end{align}
which has Jordan type $(n,n)$.
We call it {\it $2\times n$-rectangular nilpotent element}.

Consider the cocharacter $\mu:\kk^{\times}\rightarrow T\subseteq G$ defined by 
\[\mu(t)=\diag(t^{-1},t^{-2},\dots,t^{-n},t^{-1},t^{-2},\dots,t^{-n}),\]
where $T$ is the maximal torus of $G$ of diagonal matrices.
Using $\mu$ we define the $\ZZ$-grading
\begin{align}\label{z-grad}
\fg=\bigoplus\limits_{r\in\ZZ}\fg(r)\ \ \ {\rm where}\ \ \ \fg(r):=\{x\in\fg;~\mu(t)x=t^rx~{\rm for~all}~t\in\kk^{\times}\}.
\end{align}
Since the adjoint action of $\mu(t)$ on a matrix unit is given by $\mu(t)e_{i,j}=t^{\col(j)-\col(i)}e_{i,j}$,
we have $\fg(r)={\rm span}\{e_{i,j};~\col(j)-\col(i)=r\}$.

We define the following subalgebras of $\fg$:
\begin{align}\label{subalgebras}
\fp:=\bigoplus\limits_{r\geq 0}\fg(r),~\fh:=\fg(0),~{\rm and}~\fm:=\bigoplus\limits_{r<0}\fg(r).
\end{align}
Then $\fp$ is a parabolic subalgebra of $\fg$ with Levi factor $\fh$ and $\fm$ is the nilradical of the opposite
parabolic to $\fp$.
We see that $\fh$ is isomorphic to the direct sum of $n$ copies of $\mathfrak{gl}_2$.

The finite $W$-algebra $U(\fg, e)$ associated to $e$ is a subalgebra of $U(\fp)$ generated by elements denoted as follows
\begin{align}\label{w algebra generator}
\{d_i^{(r)};~1\leq i\leq 2, \, \,  r>0\}\cup\{e^{(r)}, f^{(r)};~r>0\}.
\end{align}
These elements were defined by remarkable formulas, given in \cite[\S 9]{BK08};
see also \cite[\S 4]{GT19}.
We note there is an abuse of notation as there generators of $U(\fg,e)$ have the same names as the generators for $Y_2$ given in \eqref{pbwbasis}, 
this overloading of notation will be justified in the next subsection.

Now let $\ft$ be the Lie algebra of $T$, and write $\{\epsilon_1,\dots,\epsilon_{2n}\}$ for the standard basis of $\ft^*$.
We define the weight $\eta\in\ft^*$ by
\begin{align}\label{def:eta}
\eta:=\sum\limits_{i=1}^n 2(i-n)(\epsilon_i+\epsilon_{i'}),
\end{align}
and we note that $\eta$ extends to a character of $\fp$.
For $e_{i,j}\in\fp$ define
\[
\tilde{e}_{i,j}:=e_{i,j}+\eta(e_{i,j}).
\]
Then by definition
\begin{equation} \label{e:Dir}
d_i^{(r)} := \sum_{s=1}^r (-1)^{r-s} \sum_{\substack{i_1,\dots,i_s \\ j_1,\dots,j_s}} (-1)^{|\{t=1,\dots,s-1 \mid \row(j_t) \le i-1\}|}
\tilde{e}_{i_1, j_1} \cdots \tilde{e}_{i_s, j_s} \in U(\fp)
\end{equation}
where the sum is taken over all $1 \le i_1,\dots,i_s,j_1,\dots,j_s \le 2n$ such that
\begin{itemize}
\item[(a)] $\col(j_1)-\col(i_1)+\dots+\col(j_s)-\col(i_s) + s = r$;
\item[(b)] $\col(i_t) \le \col(j_t)$ for each $t = 1,\dots,s$;
\item[(c)] if $\row(j_t) \ge i$, then $\col(j_t) < \col(i_{t+1})$ for each $t = 1,\dots,s-1$;
\item[(d)] if $\row(j_t) < i$ then $\col(j_t) \ge \col(i_{t+1})$ for each $t = 1,\dots,s-1$;
\item[(e)] $\row(i_1) = i$, $\row(j_s) = i$;
\item[(f)] $\row(j_t) = \row(i_{t+1})$ for each $t = 1,\dots,s-1$.
\end{itemize}
The expressions for the elements $e^{(r)}\in U(\fp)$ and $f^{(r)}\in U(\fp)$ are given
by similar formulas,
see \cite[\S 9]{BK08} or \cite[\S 4]{GT19} for more details.

Since $\fp$ is a restricted subalgebra of $\fg$.
We write $Z_p(\fp)_+$ for the ideal of $Z_p(\fp)$ generated by $\{x^p-x^{[p]};~x\in\fp\}$,
so the restricted enveloping algebra of $\fp$ is $U_0(\fp)=U(\fp)/U(\fp)Z_p(\fp)_+$ (see Section \ref{sec:reduced enveloping alg}).
Then the {\it restricted $W$-algerbra} is defined as
\[
U^{[p]}(\fg,e):=U(\fg,e)/(U(\fg,e)\cap U(\fp)Z_p(\fp)_+).
\] 
Since the kernel of the restriction of the projection $U(\fp)\twoheadrightarrow U_0(\fp)$ to $U(\fg,e)$ is 
$U(\fg,e)\cap U(\fp)Z_p(\fp)_+$,
we can identify $U^{[p]}(\fg,e)$ with the image of $U(\fg,e)$ in $U_0(\fp)$.

\subsection{$U^{[p]}(\fg,e)$ as a restricted truncated Yangian}
Let $I_{2,n}^{[p]}$ be the ideal of $\yp$ generated by $\{d_1^{(r)}+Z_p(Y_2)_{+};~r>n\}$.
The {\it restricted truncated Yangian $Y_{2,n}^{[p]}$} is defined to the quotient of $\yp$ by the ideal $I_{2,n}^{[p]}$ (\cite[(4.13)]{GT21}).
As before, we will use the same symbols $d_i^{(r)}, e^{(r)}, f^{(r)}$ for their canonical images in the quotient $Y_{2,n}^{[p]}$ and $U^{[p]}(\fg,e)$, respectively.
According to \cite[Theorem 1.1]{GT21},
the map from $Y_{2,n}^{[p]}$ to $U^{[p]}(\fg,e)$, determined by sending the generators $e^{(r)}, d_i^{(r)}, f^{(r)}$ of $Y_{2,n}^{[p]}$ to the generators \eqref{w algebra generator} of $U^{[p]}(\fg,e)$ with the same names, defined an isomorphism
\begin{align}\label{phip}
\phi^{[p]}:Y_{2,n}^{[p]}\rightarrow U^{[p]}(\fg,e).
\end{align}

\subsection{Irreducible representations for $Y_{2,n}^{[p]}$}
Recall from Theorem \ref{fd irreducible rep} each finite dimensional simple $\yp$-module has the form $L^{[p]}(\lambda_1(u),\lambda_2(u))$. 
Our next proposition determines for which of these simple modules the action of $\yp$ factors through the quotient $\yp \twoheadrightarrow Y_{2,n}^{[p]} $.

\begin{Proposition}\label{simple module y2np}
Let $\lambda(u)=(\lambda_1(u), \lambda_2(u))$ be restricted.
Then $L^{[p]}(\lambda(u))$ factors to a module for $Y_{2,n}^{[p]}$ if and only if $\lambda_1(u)$ and
$\lambda_2(u)$ are polynomials in $u^{-1}$ of degree $\le n$.
\end{Proposition}
\begin{proof}
Suppose that both $\lambda_1(u)$ and $\lambda_2(u)$ are restricted polynomials.
Thanks to Theorem \ref{two poly are tensor},
$L^{[p]}(\lambda(u))$ is isomorphic to some tensor product of evaluation modules which has at most $n$ tensor factors.
Using the comultiplication \eqref{hopfstructure of y2}, 
we see that every generator $t_{i,j}^{(r)}$ with $r>n$ acts on the module as the zero operator.
Note that $t_{1,1}^{(r)}=d_1^{(r)}$ \eqref{t1122 by d}.
 We thus obtain $I_{2,n}^{[p]}.L^{[p]}(\lambda(u))=(0)$.

On the other hand, Proposition \ref{another def of zpl} implies that $L^{[p]}(\lambda(u))$ is generated by $\zeta$ and $t_{i,i}(u)\zeta=\lambda_i(u)\zeta$.
Now our assertion follows from the fact that $t_{i,i}^{(r)}=0$ in $Y_{2,n}^{[p]}$ for $r>n$ (\cite[Corollary 3.6]{GT21}).
\end{proof}
In conjunction with the isomorphism $\phi^{[p]}$ from \eqref{phip}
this also shows the following:
\begin{Corollary}
The isomorphism classes of finite dimensional irreducible representations of the restricted $W$-algerbra $U^{[p]}(\fg,e)$ are parameterized by the set 
\[
\big\{(\lambda_1(u),\lambda_2(u));~\prod\limits_{j=1}^p\lambda_i(u-j+1)=1,~\deg\lambda_i(u)\leq n,~i=1,2\big\}.
\]
\end{Corollary}

\section{Applications: representations of the general linear algebra}
\label{Section:modforREA}
In this section, we continue to use the notation from Section \ref{Section:W-algebras}.
We will determine the irreducible modules for the reduced enveloping algebra of $\fg$ associated with the $2\times n$-rectangular nilpotent element.
\subsection{Baby Verma modules for $U^{[p]}(\fg,e)$}
We recall that the grading $\fg=\bigoplus_{i\in\ZZ}\fg(i)$ from \eqref{z-grad} and the notation $\fh=\fg(0)$ and $\fp=\bigoplus_{i\geq 0}\fg(i)$ from \eqref{subalgebras}.
Note that $\fh$ is reductive and is isomorphic to the $n$ copies of $\gl_2$.
We let $\fb_{\fh}$ be the Borel subalgebra of $\fh$ with basis 
\begin{align}\label{basis bh}
 \{e_{i,i};~1\leq i\leq n\}\cup\{e_{i',i'};~1\leq i\leq n\}\cup\{e_{i,i'};~1\leq i\leq n\},
\end{align}
which is the direct sum of the Borel subalgebras of upper triangular matrices in each of the $\gl_2$ summands of $\fh$.
  
Given two $n$-tuples $\alpha:=(\alpha_1,\dots,\alpha_n)$ and $\beta:=(\beta_1,\dots,\beta_n)$ of elements of $\FF_p$,
we define the weight $\lambda_{\alpha,\beta}\in\ft^*$ by
\[
\lambda_{\alpha,\beta}:=\sum\limits_{i=1}^n\alpha_i\epsilon_i+\sum\limits_{i=1}^n\beta_i\epsilon_{i'}.
\]
We define $\kk_{\alpha,\beta}=\kk.1_{\alpha,\beta}$ to be the $1$-dimensional $\ft$-module on which $\ft$ acts via 
$\lambda_{\alpha,\beta}-\eta$, where we recall that $\eta$ is defined in \eqref{def:eta}.
It is obvious that $\kk_{\alpha,\beta}$ is a $U_0(\ft)$-module.
Furthermore, 
we view it as module for $U_0(\fb_{\fh})$ on which the nilradical acts trivially.
Then we define the {\it baby Verma module} 
\[Z_{\fh}(\alpha,\beta):=U_0(\fh)\otimes_{U_0(\fb_{\fh})}\kk_{\alpha,\beta}.\]
We put $z_{\alpha,\beta}:=1\otimes 1_{\alpha,\beta}$.
We may view $Z_{\fh}(\alpha,\beta)$ as a $U_0(\fp)$-module on which the nilradical $\oplus_{i>0}\fg(i)$ of $\fp$ acts trivially.
As the restricted $W$-algebra $U^{[p]}(\fg,e)$ is a subalgebra of $U_0(\fp)$,
we restrict $Z_{\fh}(\alpha,\beta)$ to $U^{[p]}(\fg,e)$ and write $\overline{Z}_{\fh}(\alpha,\beta)$ for the restriction and $\overline{z}_{\alpha,\beta}$ for $z_{\alpha,\beta}$ viewed as an element of $\overline{Z}_{\fh}(\alpha,\beta)$.

\begin{Remark}\label{rk5}
In \cite{GT21}, the authors defined the weight $\rho_{\fh}$.
In our situation, the weight $\rho_{\fh}$ is just equal to $-\sum\limits_{i=1}^n\epsilon_{i'}$.
They defined $1$-dimensional $\ft$-module on which $\ft$ acts via 
$\lambda_{\alpha,\beta}-\eta-\rho_{\fh}$.
Equivalently, the $n$-tuple $\beta$ is replaced by $(\beta_1-1,\dots,\beta_n-1)$.
\end{Remark}

We let $e_r$ denote the $r$th elementary symmetric polynomial.
The proof of the following result is based on \cite[Lemma 5.6]{GT21}.
\begin{Lemma}\label{ed act on baby verma}
Let $\alpha:=(\alpha_1,\dots,\alpha_n)$ and $\beta:=(\beta_1,\dots,\beta_n)$ be elements of $\FF_p$. Then
\begin{enumerate}
\item[(a)] $e^{(r)} \overline{z}_{\alpha,\beta}=0$ for all $r>0$;
\item[(b)] $d_1^{(r)} \overline{z}_{\alpha,\beta}=e_r(\alpha_1,\dots,\alpha_n)\overline{z}_{\alpha,\beta}$ for all $0<r\leq n$; and 
\item[(c)] $d_2^{(r)} \overline{z}_{\alpha,\beta}=e_r(\beta_1,\dots,\beta_n)\overline{z}_{\alpha,\beta}$ for all $0<r\leq n$.
\end{enumerate}
\end{Lemma}
\begin{proof}
Instead of the construction of $Z_{\fh}(\alpha,\beta)$,
we define the Verma module $M_{\fh}(\alpha,\beta):=U(\fh)\otimes_{U(\fb_{\fh})}\kk_{\alpha,\beta}$ and write 
$m_{\alpha,\beta}:=1\otimes 1_{\alpha,\beta}$.
We can inflate it to a $U(\fp)$-module and then restrict it to $U(\fg,e)\subseteq U(\fp)$.
We write  $\overline{M}_{\fh}(\alpha,\beta)$ for the restriction and $\overline{m}_{\alpha,\beta}$ for $m_{\alpha,\beta}$ viewed as an element of $\overline{M}_{\fh}(\alpha,\beta)$.
There is a surjective homomorphism $M_{\fh}(\alpha,\beta)\twoheadrightarrow Z_{\fh}(\alpha,\beta);~m_{\alpha,\beta}\mapsto z_{\alpha,\beta}$ of $U(\fp)$-modules. As $U(\fg,e)\subseteq U(\fp)$, this gives a surjective homomorphism
\[\overline{M}_{\fh}(\alpha,\beta)\twoheadrightarrow \overline{Z}_{\fh}(\alpha,\beta);~\overline{m}_{\alpha,\beta}\mapsto \overline{z}_{\alpha,\beta}\]
of $U(\fg,e)$-modules.
Now \cite[Lemma 5.6(a)]{GT21} and its proof imply (a).
For (b) and (c), this follows from  \cite[Lemma 5.6(b)]{GT21} in conjunction with the earlier observation (Remark \ref{rk5}).
\end{proof}

We denote by $L_{\fh}(\alpha,\beta)$ the unique simple quotient of the baby Verma module  
$Z_{\fh}(\alpha,\beta)$ (see \cite[10.2]{Jan97}).
Recall that $\fh\cong\gl_2^{\oplus n}$.
For $1\leq i\leq n$, 
we write $\fg_i$ for the $i$th $\gl_2$ corresponding to the $i$th column.
It follows that $\fg_i$ has basis $\{e_{i,i}, e_{i',i'}, e_{i,i'}, e_{i',i}\}$.
For each $i$,
we have $(\lambda_{\alpha,\beta}-\eta)(e_{i,i})=\alpha_i+2(n-i)$ and 
$(\lambda_{\alpha,\beta}-\eta)(e_{i',i'})=\beta_i+2(n-i)$ .
Consequently, we obtain
\begin{align}\label{lhab}
L_{\fh}(\alpha,\beta)\cong L(\alpha_1+2(n-1),\beta_1+2(n-1))\otimes\cdots\otimes L(\alpha_n,\beta_n),
\end{align}
where $L(\alpha_i+2(n-i),\beta_1+2(n-i))$ is the irreducible $U_0(\fg_i)$-module
of dimension $[\alpha_i-\beta_i]+1$ (see Section \ref{S:evaluation module}).
As before, we restrict $L_{\fh}(\alpha,\beta)$ to $U^{[p]}(\fg,e)$ and write $\overline{L}_{\fh}(\alpha,\beta)$ for the restriction. 
We denote by $\overline{l}_{\alpha,\beta}$ the image of  $\overline{z}_{\alpha,\beta}$ in $\overline{L}_{\fh}(\alpha,\beta)$.
Also, we can view $\overline{L}_{\fh}(\alpha,\beta)$ as a $Y_{2,n}^{[p]}$-module via the isomorphism \eqref{phip}.

Given two $n$-tuples $\alpha:=(\alpha_1,\dots,\alpha_n)$ and $\beta:=(\beta_1,\dots,\beta_n)$ of elements of $\FF_p$, 
we define
\begin{eqnarray}
\label{ham}
\begin{array}{l}
\lambda_1(u)=(1+\alpha_1u^{-1})\cdots (1+\alpha_n u^{-1}),\\
\lambda_2(u)=(1+\beta_1u^{-1})\cdots (1+\beta_n u^{-1}).
\end{array}
\end{eqnarray}


Recall that the notion of a non-decreasing pair $(\alpha, \beta)$ was defined in Section~\ref{S:evaluation module}.

\begin{Theorem}\label{labconglhab}
Let $\alpha, \beta$ be as above. If $(\alpha, \beta)$ is non-decreasing
then the irreducible module $L^{[p]}(\lambda(u)) = L^{[p]}(\lambda_1(u), \lambda_2(u))$ is isomorphic to $\overline{L}_{\fh}(\alpha,\beta)$.
\end{Theorem}
\begin{proof}
Thanks to Lemma \ref{ed act on baby verma},
there is a well-defined homomorphism
\[
Z^{[p]}(\lambda_1(u),\lambda_2(u))\rightarrow \overline{L}_{\fh}(\alpha,\beta);~1_{(\lambda_1(u),\lambda_2(u))}\mapsto \overline{l}_{\alpha,\beta}
\]
of $\yp$-modules.
We let $\yp.\overline{l}_{\alpha,\beta}$ be the cyclic submodule of 
$\overline{L}_{\fh}(\alpha,\beta)$ generated by $\overline{l}_{\alpha,\beta}$.
As $L^{[p]}(\lambda_1(u), \lambda_2(u))$ is the simple quotient of 
$Z^{[p]}(\lambda_1(u),\lambda_2(u))$, there results a surjective homomorphism
\begin{align}\label{surjhom}
\overline{L}_{\fh}(\alpha,\beta)\supseteq Y_2^{[p]} \overline{l}_{\alpha,\beta}\twoheadrightarrow L^{[p]}(\lambda_\alpha(u), \lambda_\beta(u)).
\end{align}
A consecutive application of Theorem \ref{two poly are tensor} and Proposition \ref{tensor product irreducible}
implies
\[
L^{[p]}(\lambda_1(u), \lambda_2(u))\cong L(\alpha_1,\beta_1)\otimes\cdots\otimes L(\alpha_n,\beta_n)
\]
as $\yp$-modules. We note that the evaluation module $L(\alpha_i,\beta_i)$ has dimension $[\alpha_i-\beta_i]+1$ (see Subsection \ref{S:evaluation module}).
As a result,
 $\dim_\kk L^{[p]}(\lambda_1(u), \lambda_2(u))=([\alpha_1-\beta_1]+1)\cdots([\alpha_n-\beta_n]+1).$
Moreover,  by \eqref{lhab} one has $\dim L^{[p]}(\lambda_1(u), \lambda_2(u))=\dim\overline{L}_{\fh}(\alpha,\beta)$ and \eqref{surjhom} ensures that 
$L^{[p]}(\lambda_1(u), \lambda_2(u))\cong\overline{L}_{\fh}(\alpha,\beta)$.
\end{proof}

\subsection{Simple $U_\chi(\fg)$-modules}
\label{S:simple Uchig mods}
We pick $e\in\fg$ the $2\times n$-rectangular nilpotent element from \eqref{def e}.
Let $(\cdot,\cdot):\fg\times \fg\rightarrow\kk;~(x,y)\mapsto {\rm tr}(xy)$ denote the trace from associated to the natural representation of $\fg$.
Let $\chi\in\fg^*$ be the element dual to $e$ via the trace from $(\cdot,\cdot)$,
i.e. $\chi=(e,\cdot)$.
The reduced enveloping algebra is defined by
\[
U_\chi(\fg)=U(\fg)/J_{\chi}=U(\fg)/(x^p-x^{[p]}-\chi(x)^p;~x\in\fg).
\]
To describe the simple $U_\chi(\fg)$-modules,
we require some notations.
Since $e\in\fg(1)$, we have that $\chi$ vanishes on $\fg(k)$ for $k\neq -1$.
Therefore $\chi$ restricts to a character of $\fm$, 
so that $\chi$ defines a one dimensional representation $\kk_{\chi}=\kk.1_{\chi}$ of $U_\chi(\fm)$.
Following Premet \cite{Pre02} we define the {\it restricted Gelfand–Graev module} to be
\begin{align}\label{Qchi}
Q^{\chi}:=U_\chi(\fg)\otimes_{U_\chi(\fm)}\kk_\chi\cong U_\chi(\fg)/U_{\chi}(\fg)\fm_{\chi},
\end{align}
where we recall that $\fm_\chi=\{x-\chi(x);~x\in\fm\}$. 
Note that $Q^{\chi}$ is a left $U_\chi(\fg)$-module and a right $U^{[p]}(\fg,e)$-module.

Premet's equivalence is a Morita equivalence between $U^{[p]}(\fg,e)$-{\rm mod} and $U_\chi(\fg)$-{\rm mod} (see \cite[Theorem 2.4]{Pre02}, \cite[Proposition 4.1]{Pre07} and \cite[Theorem 2.4]{GT21}). More precisely, the functor from $U^{[p]}(\fg,e)$-{\rm mod} to $U_\chi(\fg)$-{\rm mod} given by
\[M\mapsto Q^{\chi}\otimes_{U^{[p]}(\fg,e)}M\]
is an equivalence of categories, with quasi-inverse given by $V\mapsto V^{\fm_\chi}:=\{v\in V;~\fm_\chi v=(0)\}$. Furthermore, $Q^\chi$ is a free right $U^{[p]}(\fg,e)$-module.

Taking into account Theorem \ref{labconglhab}, we obtain our final result:
\begin{Theorem}
\label{T:modgllofa}
\begin{enumerate}
\setlength{\itemsep}{4pt}
\item Every $U_\chi(\gl_{2n})$-module is isomorphic to one of the modules
$$\{ Q^{\chi}\otimes_{U^{[p]}(\fg,e)}\overline{L}_{\fh}(\alpha,\beta);~ \alpha, \beta \in \FF_p^n, \, \, (\alpha, \beta) \text{ non-decreasing}\}.$$
\item If $(\alpha, \beta)$ and $(\alpha', \beta')$ are both non-decreasing, they give rise to $(\lambda_1(u), \lambda_2(u))$ and $(\lambda_1'(u), \lambda_2'(u))$ in accordance with \eqref{ham}. Then the following are equivalent:\vspace{4pt}
\begin{itemize}
\item[(i)] $Q^{\chi}\otimes_{U^{[p]}(\fg,e)}\overline{L}_{\fh}(\alpha,\beta) \cong Q^{\chi}\otimes_{U^{[p]}(\fg,e)}\overline{L}_{\fh}(\alpha',\beta')$;\vspace{4pt}
\item[(ii)] $\lambda_1(u) = \lambda_1'(u)$ and $\lambda_2(u) = \lambda_2'(u)$.\vspace{4pt}
\end{itemize}
In particular, simple $U_\chi(\gl_{2n})$-modules are parametrised by pairs of restricted polynomials of degree $n$.
\item $Q^{\chi}\otimes_{U^{[p]}(\fg,e)}\overline{L}_{\fh}(\alpha,\beta)$
is a simple $U_\chi(\fg)$-module with
$$\dim_{\kk}Q^{\chi}\otimes_{U^{[p]}(\fg,e)}\overline{L}_{\fh}(\alpha,\beta)=p^{2n^2-2n}([\alpha_1-\beta_1]+1)\cdots([\alpha_n-\beta_n]+1).$$

\end{enumerate}
\end{Theorem}

It would be very interesting to relate this parameterisation to the one appearing in \cite[Proposition~10.5]{Jan97}.

\end{document}